\documentclass[11pt,reqno,a4paper]{amsart}
\usepackage[utf8]{inputenc}
\usepackage[headings]{fullpage}

 \usepackage{comment} 
 
\usepackage[english]{babel}
\usepackage[T1]{fontenc}         
\usepackage{lmodern}             
\usepackage{amsmath,amsthm,amsfonts,amssymb,bm}           

\usepackage{todonotes}

\usepackage{mathrsfs}

\usepackage{graphicx}
\usepackage{tikz-cd}
\usepackage{rotating}

\definecolor{blue-grey}{HTML}{4A90E2}
\usepackage{url}
\usepackage[colorlinks,backref=page,allcolors=blue-grey]{hyperref}
\usepackage[alphabetic,backrefs,lite]{amsrefs}
\usepackage[capitalize,nameinlink,noabbrev]{cleveref} 
\crefname{page}{page}{pages}

\renewcommand{\geq}{\geqslant}
\renewcommand{\leq}{\leqslant}
\renewcommand{\ge}{\geqslant}
\renewcommand{\le}{\leqslant}

\newtheorem{thm}{Theorem}[section]
\newtheorem{cor}[thm]{Corollary}
\newtheorem{lemma}[thm]{Lemma}

\newtheorem{thmintro}{Theorem}
\newtheorem{proposition}[thm]{Proposition}

\theoremstyle{definition}
\newtheorem{definition}[thm]{Definition}

\theoremstyle{remark}
\newtheorem{remark}[thm]{Remark}
\newtheorem{notation}[thm]{Notation}

\DeclareMathOperator{\Sym}{\mathrm{Sym}}

\DeclareMathOperator{\CEff}{\mathrm{Eff}} 
\DeclareMathOperator{\Hom}{\mathrm{Hom}} 
\DeclareMathOperator{\Spec}{\mathrm{Spec}} 

\DeclareMathOperator{\Pic}{\mathrm{Pic}} 
\DeclareMathOperator{\Div}{\mathrm{Div}} 
\DeclareMathOperator{\PPic}{\mathbf{Pic}} 
\DeclareMathOperator{\DDiv}{\mathbf{Div}} 

\DeclareMathOperator{\pr}{\mathrm{pr}} 
\DeclareMathOperator{\rk}{\mathrm{rk}}  
\DeclareMathOperator{\ddiv}{\mathrm{div}}  

\newcommand{\KVar}[1]{\mathrm{K_0} \mathbf{Var}_{#1} } 

\newcommand{\Kapr}{\mathrm{Kapr}}

\newcommand{\C}{\mathscr{C}}
\newcommand{\LL}{\mathbf{L}}

\newcommand{\PP}{\mathbb{P}}
\newcommand{\E}{\mathcal{E}}
\newcommand{\OK}{\mathcal{O}}

\newcommand{\Nef}{\mathrm{N}^1}

\begin{document}

\title[Motivic counting on $DP_5$]{Motivic counting of curves\\ on split quintic del Pezzo surfaces}

\author[C. Bernert]{Christian Bernert}
\address{ISTA \\ Am Campus 1 \\ 3400 Klosterneuburg \\ Austria}
\email{christian.bernert@ist.ac.at}
\author[L. Faisant]{Loïs Faisant}
\address{KU Leuven, Department of Mathematics, B-3001 Leuven, Belgium}
\email{lois.faisant@kuleuven.be}
\author[J. Glas]{Jakob Glas}
\address{Leibniz Universität Hannover\\
Welfengarten 1\\
30167 Hannover\\
Germany}
\email{glas@math.uni-hannover.de}


\begin{abstract}
    We prove the ``all-the-heights'' version of the Batyrev--Manin--Peyre conjecture for split quintic del Pezzo surfaces, 
    both for counting rational points over global function fields in positive characteristic 
    and for the motivic version over a general base field. 
\end{abstract}

\maketitle


\setcounter{tocdepth}{1}
\tableofcontents


\section*{Introduction}

Let $X$ be a smooth split del Pezzo surface of degree $5$ over a field $k$ and let $U\subset X$ be the complement of the ten lines. When $k$ is a number field, a conjecture of Manin 
\cite{batyrev1990nombre,franke-manin-tschinkel1989rational}
predicts that for $B>0$ the counting function 
\[
N_U(B)= \#\{x\in U(k)\mid H(x)\leq B\}
\]
where $H\colon X(k)\to \mathbf{R}_{\geq 0}$ is an anticanonical height function, satisfies the asymptotic formula 
\begin{equation}\label{Eq: Asymp.Manin}
N_U(B)\sim c_{H,\text{Pey}}B(\log B)^4 \quad\text{as }B\to \infty.
\end{equation}
Here, the leading constant $c_{H,\text{Pey}}$ has been given a conjectural interpretation by Peyre~\cite{peyre1995hauteurs}. The conjecture of Manin--Peyre applies to arbitrary smooth Fano varieties and the particular case when $X$ is a split quintic del Pezzo surface has been the subject of extensive study. Indeed, Tschinkel in his PhD thesis~\cite{TschinkelThesis} proved an upper bound of the shape $N_U(B)\ll_\varepsilon B^{1+\varepsilon}$.  

In 1993, Salberger\footnote{Announced at the lecture ``Counting rational points on del Pezzo surfaces of degree 5'' in Bern, 1993}
used universal torsors to give \emph{upper bounds} for $N_U(B)$ matching the order of growth predicted by \eqref{Eq: Asymp.Manin}.
This effort culminated in work of de la Bret\`eche~\cite{delaBreteche2002dP5}, where he obtained the conjectured asymptotic formula \eqref{Eq: Asymp.Manin} over $\mathbf{Q}$. Recently, this was generalised to arbitrary number fields by Derenthal and the first named author~\cite{bernert-derenthal2024dp5}. When $X$ is a quintic del Pezzo surface that is not necessarily split, there are further examples for which \eqref{Eq: Asymp.Manin} holds by the works of de la Bret\`eche--Fouvry~\cite{BretecheFouvry} and Heath-Brown--Loughran~\cite{heathbrownLoughran}.

In the last years, the Manin--Peyre conjecture has been generalised in various directions. This includes, but is not limited to, an ``all-the-heights'' version by Peyre~\cite{peyre2021beyond} 
as well as geometric~\cite{lehmann-tanimoto2019geometric-Manin}
and motivic analogues~\cite{bourqui2009produit-eul-mot,bilu2023MAMS,faisant2025motivic-distribution}
of the Manin--Peyre conjectures. 
In this work, we are concerned with establishing the ``all-the-heights'' version of the Manin--Peyre conjecture both in the function field and in the motivic setting. 
To describe our results in more detail, let $\C$ be a smooth, projective, geometrically integral curve of genus $g$ and let $\bm{d}\in \Pic(X)^\vee$ be in the dual of the cone of effective divisors. 
To such data, we can associate the space of morphisms
\[
\Hom^{\bm{d}}_k(\mathscr{C}, X)_U = \{f\colon \C\to X 
\mid 
f( \mathscr C ) \cap U\neq \emptyset
\text{ and } f_*(\C)=\bm{d}\}.
\]


Our first result provides a precise asymptotic formula for the number of $k$-rational points of this moduli space when $k = \mathbf F_q$ is a finite field with a power saving error term. For any variety $Y$ over $\mathbf F_q$, we denote by $\#_{\mathbf F_q} Y$ its number of $\mathbf F_q$-points. 

\begin{thmintro}\label{Thm:TheTheorem.Fq(C)}

Assume that $k = \mathbf F_q$. Let $\omega_X^\vee$ be the anticanonical sheaf of $X$.
Then, there exists $\eta > 0 $ such that
for any $\bm d \in \Pic ( X )^\vee$ lying in the dual of the effective cone $\CEff ( X )$,
\[
\#_{\mathbf F_q}
\Hom^{\bm d}_{\mathbf F_q} ( \mathscr C , X )_U  
= c \cdot q^{d+2(1-g)}
    +
    O ( q^{d - \eta \mathrm{dist}( \bm d, \partial \CEff ( X ) ^\vee ) }),
\]
where $d=\bm{d}\cdot \omega_X^\vee $ is the anticanonical degree of $\bm{d}\in \Pic(X)^\vee$,
$\mathrm{dist}( \bm d, \partial \CEff ( X ) ^\vee )$ is the distance of $\bm d $ to the boundary of $\CEff ( X ) ^\vee $,
and
\[
c
=
 \left ( 
    \frac{q^{-g}
     \#_{\mathbf F_q} \PPic^0 ( \mathscr C ) 
     }
     {1 -  q^{-1}}
      \right )^{5}
\prod_{v \in | \mathscr C |}
\left ( 
1 - q_v^{-1}
\right )^5
\left ( 
    1 + 5 q_v^{-1} + q_v^{-2}
\right )
\]
where $ | \mathscr C |$ is the set of closed points of $\mathscr C $ and $ q_v$ is the cardinality of the residue field at $v \in | \mathscr C |$.
\end{thmintro}
Our result provides the first positive answer to the ``all-the-heights'' conjecture made by Peyre~\cite{peyre2021beyond} for quintic del Pezzo surfaces over \emph{any} global field.
In particular, the constant $c$ matches the predictions. 

In addition, when working over function fields, much less is known about the Manin--Peyre conjecture for del Pezzo surfaces compared to the situation over number fields. 
There are upper bounds for the analogue of the counting function $N_U(B)$ from the works of Hochfilzer and the third named author~\cite{glas-hochfilzer2024low-deg-dP, glas2025lineargrowthmodulispaces}  that are not sharp
and obtaining an asymptotic formula is still an open problem. 
However, for del Pezzo surfaces of higher degree much more is known thanks to the work of Bourqui~\cite{bourqui2003toriques-deployees,bourqui2011toriques-non-deployes}. 
The analogue of the Manin--Peyre conjecture in the function field setting goes back to work of Peyre~\cite{peyre2012drapeaux} and Bourqui~\cite{bourqui2011toriques-non-deployes,bourqui2011Manin-geometrique-quadriques-intrinseque}. For recent work in this direction see also~\cite{lehmann2026geometricmaninsconjecturecharacteristic}.
When $\mathscr C = \mathbf P^1_k$,
the present work can be compared to Tanimoto's recent work \cite{tanimoto2025DP5P1conicbundles}, 
where a similar result is obtained only when $q$ is assumed to be sufficiently large with respect to the distance of $\bm{d}$ to $\partial \text{Eff}(X)^\vee$, using a homological stability approach recently developed in~\cite{das2025homologicalstabilitymaninsconjecture}.



\subsection*{Motivic Manin--Peyre principle for quintic del Pezzo surfaces}
We now describe our main result in the motivic setting. Let $k$ be an arbitrary field and $\KVar{k}$ be the Grothendieck ring of varieties, whose generators are isomorphism classes of varieties over $k$, modulo cut and paste relations. 
Let $\mathbf L_k$ be the class of the affine line over $k$, which we invert to get the localisation $\mathscr M_k$. This localisation admits a natural completion with respect to the filtration by the virtual dimension, which we denote by 
$\widehat{\mathscr M_k}^{\dim}$. 

It is then natural to consider directly the class of $\Hom^{\mathbf{d}}_k(\mathscr{C}, X)_U
$ in $\mathscr M_k$ 
and study its stabilisation properties. 
In general, the question of the behaviour of this class is the subject of the motivic analogue of Manin's program~\cite{faisant2025motivic-distribution}.  
In parallel to our asymptotic counting result, we prove that the motivic Manin--Peyre principle~\cite[Question 1]{faisant2025motivic-distribution} holds in our situation.
\begin{thmintro}\label{Th: TheTheorem.Motivic}
Assume that $\mathscr C$ admits a divisor of degree one over $k$. Then
\[
\left [
\Hom^{\bm d}_k ( \mathscr C , X )_U 
\right ] 
\mathbf L_k^{- d - 2 ( 1 - g )} 
\in \mathscr M_k
\]
tends to
\[
\left ( 
    \frac{
     \left [ \PPic^0 ( \mathscr C ) \right ] \mathbf L_k^{-g}
     }
     {1 - \mathbf L_k^{-1}}
      \right )^{5}
\prod_{v \in \mathscr C}
\left ( 
1 - \mathbf L_v^{-1}
\right )^5
\left ( 
    1 + 5 \mathbf L_v^{-1} + \mathbf L_v^{-2}
\right ) \in \widehat{\mathscr M_k}^{\dim}
\]
as $\min ( d_i )$ tends to infinity.
\end{thmintro}

The product over $\mathscr C$ appearing in the limit is taken in the sense of Bilu's notion from~\cite[Chap. 3]{bilu2023MAMS}. It is given by the evaluation at $\LL^{-1}_k$ of a certain formal power series whose coefficients are defined in terms of symmetric products of varieties. In our situation, it is already convergent at $\LL^{-1}$ with respect to the dimensional topology. In general, one needs to work with a finer filtration given by the weight of the underlying Hodge structures; see \cite[Chap. 4]{bilu2023MAMS} and \cite[\S 3.1]{faisant2025motivic-distribution}. 

Sticking to del Pezzo surfaces, the motivic Manin--Peyre principle was previously known only for the smooth and split toric ones by \cite{bourqui2009produit-eul-mot,bilu-das-howe2022zeta-stat,faisant2025motivic-distribution} when $\mathscr C = \mathbf P^1_k$ and \cite{faisant2025univ-torsors} for any $\mathscr C$ with a degree one $k$-divisor. 

Our result is compatible with 
the prediction from \cite{faisant2025motivic-distribution}, in particular 
the local density $\left ( 
    1 + 5 \mathbf L_v^{-1} + \mathbf L_v^{-2}
\right ) = \frac{[X_v]}{\mathbf L_v^2}$ 
can be interpreted as the motivic measure of the arc space of $X$.

A corollary of \cref{Th: TheTheorem.Motivic} is that if $\min ( d_i ) $ is large enough, then the moduli space $\Hom^{\bm d}_k ( \mathscr C , X )_U$ is geometrically irreducible of expected dimension $d+2(1-g)$. 

\subsection*{Method and outline.}
We prove both \cref{Thm:TheTheorem.Fq(C)} and \cref{Th: TheTheorem.Motivic} in close parallel. The starting point is a passage to the universal torsor of $X$, which is classically known to be the affine cone of the Grassmannian $\mathbb{G}(3,5)$ under the Pl\"ucker embedding by work of 

Skorobogatov~\cite{skorobogatov1993theorem-Enriques-Swinnerton-Dyer}. 
In line with previous works of Schanuel~\cite{schanuel1979heights-in-nf}
and Peyre~\cite{peyre1995hauteurs}, 
Salberger was the first to apply the universal torsor method to study the Manin--Peyre conjecture for quintic del Pezzo surfaces. He then developed the universal torsor method in a systematic way~\cite{salberger1998tamagawa} 
and applied it to split toric varieties. 
The universal torsor method also plays a crucial role in the works of de la Bret\`eche~\cite{delaBreteche2002dP5} and Bernert--Derenthal~\cite{bernert-derenthal2024dp5}. 

Thanks to recent work of the second named author~\cite{faisant2025univ-torsors}, 
we can express the class of $\Hom_k^{\bm{d}}(\C, X)$ in $\KVar k$
in terms of global sections of line bundles satisfying the Plücker equations subject to certain coprimality conditions. This parameterisation is recalled in \cref{Se: LiftingUT}.
In the work of Bernert--Derenthal~\cite{bernert-derenthal2024dp5}, one then proceeds to rephrase this in terms of counting points on a family of lattices. 

We take a similar route with two key differences: in our setting, the natural analogue of lattice points in bounded regions are global sections of vector bundles.
In \cite{bernert-derenthal2024dp5}, lower bounds for the first successive minima of these lattices are given, 
which are translated into estimates using the theory of $o$-minimal structures through the work of Barroero--Widmer~\cite{barroero-widmer2014counting-lattice-points-o-minimal}. Instead of using tools from the geometry of numbers, we work with the theory of slopes and Harder--Narasimhan filtrations of vector bundles, whose basic properties we recall in \cref{Se: VectorBundles}. For a detailed comparison between those two concepts see the recent work of Bost--Paulin~\cite{bost2025minimasuccessifsdesreseaux}.

When treating the error terms in \cref{Se: ErrorTerms}, we show that the Harder--Narasimhan filtrations of the corresponding vector bundles are balanced on average. In the number field setting, this would correspond to showing that the successive minima of these lattices are on average of the expected size. This leads to significant improvements in the error terms we obtain and it seems likely that one can obtain an analogue of the Manin--Peyre conjecture over function fields with error terms that supersede what is currently known in characteristic 0 \cite{browning2022revisitingdp5}. We intend to return to this problem in future work.

\subsection*{Acknowledgements}

LF acknowledges partial funding from the European Union’s Horizon 2020 research and innovation
programme under the Marie Skłodowska-Curie grant agreement No 101034413,
as well as 
partial support by the KU Leuven IF grant C16/23/010.
This work was presented at the ``Cox rings and applications'' research school at CIRM (March 2026)
thanks to the financial support of the ANR project AdAnAr. 

The authors thank Tim Browning and Ulrich Derenthal for several discussions around the matter of this article.


\renewcommand{\theequation}{\thesection.\arabic{equation}}
\counterwithin*{equation}{section}


\section{Geometric preliminaries}\label{Se: GeoPrelim}


\subsection{Decomposition of the nef cone}

In this section, we explain how to choose our coordinates in a way allowing a uniform counting procedure later on. In particular, it replaces the use of symmetry in~\cite[Section~3.2]{bernert-derenthal2024dp5}.

Let $X_d$ be a split del Pezzo surface of degree $5\leq d\leq 7$ over $k$. We denote by $\Pic(X_d)$ the Picard group of $X_d$ and by $\Nef(X_d)=\Pic(X_d)\otimes \mathbf{R}$ the space of $\mathbf{R}$-divisors up to numerical equivalence. Let $\overline{\text{Eff}}^1(X_d)\subset \Nef(X_d)$ be the pseudoeffective cone of divisors. The intersection pairing 
\[
\cdot \colon \Nef(X_d)\times \Nef(X_d)\to \mathbf{Z}
\]
identifies $\Pic(X_d)$ as well as $\Nef(X_d)$ with their dual and we write $\text{Nef}_1(X_d)\subset \Nef(X_d)$ for the nef cone of $X_d$, which by definition is the dual cone of $\overline{\text{Eff}}^1(X_d)$. It is well known that $\overline{\text{Eff}}^1(X_d)$ is generated by classes of $(-1)$-curves and $\text{Nef}_1(X_d)$ by fibers of conic fibrations and pullbacks of the hyperplane class under birational morphisms $X_d\to \PP^2$.

Let $\alpha$ be a nef class on $X_5$. Since the effective cone of $X_d$ is generated by classes of $(-1)$-curves, this means that $\alpha \cdot L\geq 0$ for every irreducible $(-1)$-curve $L$ of $X_d$. It will be convenient to call an irreducible $(-1)$-curve of $X_d$ simply a line. 
Let us recall the following decomposition of the nef cone of $X_d$.  

\begin{lemma}\label{Le: SmallLines}
    Let $X_d$ be a split del Pezzo surface of degree $3\leqslant d \leqslant 6$ and let $\alpha$ be a nef curve class on $X_d$. Then there exist a contraction of a $(-1)$-curve $E$ given by $\psi\colon X_d\to X_{d-1}$ to a del Pezzo surface $X_{d-1}$ of degree $d-1$, an integral nef class $\beta$ in $X_{d-1}$ and a non-negative integer $b\geq 0$ such that 
    \[
    \alpha = -bK_{X_d} +\psi^*\beta.
    \]
    In addition, we have 
    \[
    b= \min_{\substack{L\subset X_d \\ \text{line}}} \alpha\cdot L
    \]
    and $\beta, \psi, b$ are unique if $\beta$ is ample. 
\end{lemma}
\begin{proof}
    This is the content of Lemma~3.2 in~\cite{das2025homologicalstabilitymaninsconjecture} 
    except for the assertion about the minimum. 
    To see this,
    let $E'$ be another class of a line on $X$. Then by the projection formula we have 
    \[
    \alpha\cdot E'= b+(\psi^*\beta)\cdot E' = b+\beta\cdot (\psi_*E')\geq b,
    \]
    where we used that $\beta$ is nef. 
\end{proof}
The nef cone of a split del Pezzo surface $X_7$ of degree $7$ is generated by $H, F_1,F_2,$ where $H$ is the unique pullback of the class of a hyperplane when realised as the blow-up of $\PP^2$ in two points and $F_1$ and $F_2$ are the class of a fiber of the two different conic fibrations. Applying \cref{Le: SmallLines} twice gives us a sequence of contractions of $(-1)$-curves $E_1,E_2$
\[
    \psi_{57} : 
    X_5
    \overset{\psi_{56}}{\longrightarrow}
    X_6
    \overset{\psi_{67}}{\longrightarrow}
    X_7 
\]
and non-negative integers $b_1,b_2, a, c_1, c_2$ such that
\begin{equation}\label{Eq: ConeDecomp}
\alpha = -b_1K_{X_5}-b_2\psi_{56}^*K_{X_6}+\psi_{57}^*(aH+c_1F_1+c_2F_2) .
\end{equation}
In addition, this representation is unique if $b_2,a,c_1,c_2>0$ and after relabelling the conic fibrations we can assume that $c_1\leqslant c_2$. 

Suppose that $E_3',E_4'$ are the exceptional divisors of a blow-up realisation $\phi\colon X_7 \to \PP^2$ such that $F_1=H-E_3'$ and $F_2=H-E_4'$ 
and put $E_3=\psi_{57}^*E_3'$, $E_4=\psi_{57}^*E_4'$. Furthermore, let $E'_{34}$ be the class on $X_7$ of the strict transform of the line through the two blown-up points and set $E_{34}=\psi_{57}^*E'_{34}$. The composition $\phi\circ \psi_{57}\colon X_5\to \PP^2$ then realises $X_5$ as the blow-up of $\PP^2$ in four points $p_1,p_2,p_3,p_4$ such that $E_i$ is the exceptional divisor above $p_i$. 

Given four pairwise disjoint lines $E_1,E_2,E_3,E_4$, we let $\mathcal{T}_{(E_1,E_2,E_3,E_4)}$ be the subcone of the nef cone of $X_5$ generated by elements of the form $\eqref{Eq: ConeDecomp}$ with $b_1,b_2, a, c_1,c_2\geq 0$ and $c_1\leqslant c_2$. By \cref{Le: SmallLines}, every nef class lies in one of these cones and there are precisely $10\times 6\times 2 = 120$ such cones. Furthermore, $\alpha$ belongs to a unique cone if $b_2,a,c_1,c_2>0$. Let us summarise our findings in the following result, with a more convenient presentation. 
\begin{proposition}\label{Prop:ConeDecomp}
    Let $\alpha$ be a nef curve class on $X_5$. Then there are pairwise disjoint lines $E_1,E_2,E_3,E_4\subset X_5$ that are the exceptional divisors of a blow-up realisation $X_5\to \PP^2$ in $p_1,p_2,p_3,p_4\in \PP^2(k)$ such that $\alpha$ belongs to $\mathcal{T}_{(E_1,E_2,E_3,E_4)}$. In addition, if we set 
    \[
    d_i= \alpha \cdot E_i, \quad i=1,\dots, 4,
    \]
    and 
    \[
    d_{ij}=\alpha\cdot L_{ij},\quad 1\leqslant i<j\leqslant 4,
    \]
    where $L_{ij}$ is the class of the strict transform of the line through $p_i$ and $p_j$, then 
    \[
    d_1\leqslant d_2\leqslant d_3\leqslant d_4
    \]
    and 
    \[
    d_1=\min_{L\text{ line}}\alpha\cdot L, \quad d_2=\min_{\substack{L\text{ line}\\ L\cdot E_1=0}}\alpha \cdot L.
    \]
    Finally, $\alpha$ belongs to a unique cone $\mathcal{T}_{(E_1,E_2,E_3,E_4)}$ if $d_1<d_2<d_3$ and $d_2<d_{34}$. 
\end{proposition}
\begin{proof}
    Upon observing that 
    \begin{align*}
        d_1& = b_1,\\
        d_2 & =b_1+b_2,\\
        d_{34}& = b_2+b_2+a,\\
        d_3& = b_1+b_2+c_1,\\
        d_4& = b_1+b_2+c_2,
    \end{align*}
this follows from the previous discussion and \cref{Le: SmallLines}. 
\end{proof}
It will be useful to have explicit upper bounds for the $d_i$'s when $\alpha$ belongs to $\mathcal{T}_{(E_1,E_2,E_3,E_4)}$. 
\begin{cor}\label{Cor: Upperbounds.di}
    Let $\alpha\in \mathcal{T}_{(E_1,E_2,E_3,E_4)}$ and set $d=-\alpha\cdot K_{X_5}$, where $K_{X_5}$ is the canonical divisor of $X_5$. Then in the notation of \cref{Prop:ConeDecomp} we have 
    \begin{align*}
        d_1, d_2 & \leq d/5,\\
        d_1+d_2+d_3+d_4 &\leq 4d/5.
    \end{align*}
\end{cor}
\begin{proof}
    In $\Pic(X_5)$ we have 
    \[
    L_{ij}+E_i+L_{ik}+E_k+L_{kl}=-K_{X_5}
    \]
    for any choice of indices $\{i,j,k,l\}=\{1,2,3,4\}$ as can be easily verified by a direct computation. In particular, since $d_1=\min_{L\subset X_5} \alpha\cdot L$ we obtain 
    \[
    d = d_{12}+d_1+d_{13}+d_3+d_{34} \geq 5d_1
    \]
    and hence $d_1\leq d/5$. 
    Using the relations $L_{ij}=H-E_i-E_j$ and $-K_{X_5}= 3H-E_1-\cdots -E_4$, one can easily show that 
    \[
    E_2+L_{23}+E_3+L_{34}+E_4+L_{24}=-K_{X_5}+E_1,
    \] 
    therefore we also have 
    \[
    d_2+d_{23}+d_3+d_{34}+d_4+d_{24}=d+d_1 \leq 6d/5.
    \]
    Since $d_2=\min_{L\cdot E_1=0} \alpha \cdot L$ and $E_2, L_{23}, E_3, L_{34}, E_4, L_{24}$ are precisely the lines disjoint to $E_1$, this implies 
    \[
    6d_2 \leq 6d/5,
    \]
    whence $d_2\leq d/5$ as claimed. 

    Finally, we have 
    \[
    3L_{34}+2E_3+2E_4=-K_{X_5}+E_1+E_2,
    \]
    so that $d_1, d_2 \leq d_{34}$, which implies $(d_1+d_2)/2 \leq d_{34}$, and $d_1, d_2 \leq d$ give
    \[
    2d_3+2d_4 \leq d -d_1/2-d_2/2.
    \]
Therefore, 
\[
d_1+d_2+d_3+d_4 \leq d/2 +3d_1/4 + 3d_2/4 \leq d/2 + 3d_1/20 + 3d_2/20 = 4d/5
\]
as desired. 
\end{proof}

\subsection{Vector bundles on curves}
\label{Se: VectorBundles}
In this section we collect some basic facts about vector bundles on curves. 


\subsubsection{Semistable vector  bundles}
\begin{definition}
    The \emph{slope} of a non-zero vector bundle $\mathcal{E}$ on $\mathscr C$ is defined as 
    \[
    \mu ( \E ) 
    = 
    \frac{\deg (\E )}{\rk ( \E )} . 
    \]
    Moreover, a vector bundle $\E$ on $\mathscr C$
    is said to be \emph{semistable}
    if for every non-zero proper subbundle $\E' \subsetneq \E$,
    we have 
    \[
    \mu ( \E ' ) \leqslant \mu ( \E ) .
    \]
\end{definition}

\begin{remark}
    Since the zero bundle contains no non-zero proper subbundle, it is automatically semistable. 
\end{remark}

\begin{lemma}
Let 
\[
\begin{tikzcd}[sep=small]
    0 \rar  
        & 
        \E_1
        \rar 
        &
        \E_2 
        \rar 
        &
        \E_3 
        \rar 
        & 0
\end{tikzcd}
\]
be 
a short exact sequence of vector bundles.
Then 
\[
\mu(\E_2)=\mu(\E_1)+\mu(\E_3). 
\]
\end{lemma}

\begin{thm}[Harder--Narasimhan~\cite{harder-narasimhan1975cohomology}]
    Let $E$ be a vector bundle on a smooth, projective and geometrically irreducible curve $\mathscr C$. 
    There exists a unique filtration 
    \[
    0 = \E_0 \subsetneq \E_1 \subsetneq ... \subsetneq \E_r = \E
    \]
    such that 
    \begin{itemize}
        \item $\E_i / \E_{i-1}$ is semistable of slope $\lambda_i$ for $i \in \{ 1 , ... , r \}$, 
        \item with $\lambda_1 > ... > \lambda_r$.
    \end{itemize}
\end{thm}

\begin{remark}
Here $r$ is part of the datum of the filtration and \emph{does not} stand for \emph{rank}.
For example, if $\E$ is already semistable (of arbitrary rank), then $r=1$ in the statement of the theorem. 
\end{remark}

\begin{remark}
The additivity of the slope in short exact sequences implies that 
\begin{equation}\label{Eq: SumSuccMin=Slope}
\mu(\E)
= 
\mu(\E_{r-1})+\lambda_r
= 
\mu(\E_{r - 2})
+
\lambda_{r-1}
+
\lambda_r
=
\cdots 
= 
\sum_{i=1}^r\lambda_r.
\end{equation}
In particular, we obtain 
\[
\mu(\E)\leqslant r\lambda_1\leqslant \rk(\E)\lambda_1.  
\]
\end{remark}

We will make use of the following two lemmas
whose proofs are classical.

\begin{lemma}\label{Le: SlopeunderTwisting}
    If $\E$ is a vector bundle on 
    $\mathscr C$ of slope $\lambda$ and $L$ is a line bundle on $\C$,
    then 
    $\E\otimes \mathcal{L} $ has slope $\lambda + \deg(\mathcal{L})$.
    Moreover, if $\E$ is semistable, then so is $\E \otimes \mathcal{L}$.
\end{lemma}

\begin{lemma}\label{Le: SemiStablePosDegNoGlobal}
    A semistable vector bundle of negative slope admits no non-trivial global section. 
\end{lemma}

By the Riemann--Roch theorem for vector bundles, we have 
\[
h^0(\C, \E)= \rk(\E)(\mu(\E)+1-g)+h^1(\C,\E). 
\]
We are going to show that if $h^1(\C,\E)$ is large, then so is $\lambda_1$. 
We begin with the following auxiliary result. 
\begin{lemma}\label{Le: h^1not0implieslambda_rsmall}
    Suppose that $h^1(\C,\E)> 0$. Then 
    \[
    \lambda_r\leqslant 2g-2.
    \]
    In particular, 
    if $\E$ is semistable, 
    then $\mu ( \E ) > 2g - 2$
    implies $h^1 ( \mathscr C , \E ) = 0$. 
\end{lemma}
\begin{proof}
    Suppose that $h^1(\C,\E)> 0$. Then the long exact sequence in cohomology coming from the short exact sequence 
    \[
    \begin{tikzcd}[sep=small]
        0
        \rar & 
        \E_{r-1} \rar  
        & \E_r \rar 
        & \E_r/\E_{r-1} \rar 
        & 0 
    \end{tikzcd}
    \]
    implies that $h^1(\C,\E_r/\E_{r-1})\neq 0$. By Serre-duality, we have 
    \[
    h^0(\C, (\E_r/\E_{r-1})^\vee \otimes \omega_{\C})=h^1(\C, \E_r/\E_{r-1})> 0,
    \]
    where $\omega_\C$ denotes the canonical sheaf.     However, the vector bundle $(\E_r/\E_{r-1})^\vee \otimes  \omega_{\C} $ 
    is semistable of slope $-\mu(\E_r/\E_{r-1})+2g-2$. In particular, by \cref{Le: SemiStablePosDegNoGlobal} we must have 
    \[
    \lambda_{r-1}=\mu(\E_r/\E_{r-1})\leqslant 2g-2.
    \]
    Plugging this into \eqref{Eq: SumSuccMin=Slope} yields 
    \[
    \mu(\E)=\sum_{i=1}^r\lambda_i \leqslant (r-1)\lambda_1 +2g-2,
    \]
    from which the statement of the lemma follows, since $r\leqslant \rk(\E)$. 
\end{proof}
\begin{lemma}\label{Le: Upperbound.h^0.semistable}
    Suppose that $\E$ is semistable. Then 
    \[
    h^0(\C,\E)\leqslant \max\{\rk(\E)(\mu(\E)+1-g), \rk(\E)g\}.
    \]
    Moreover, if $\mu ( \E ) \geqslant 2g$, then $\E$ is generated by its global sections. 
\end{lemma}
\begin{proof}
    Note that if $\mu(\E)<0$, then the statement holds trivially, since $h^0(\C,\E)=0$ by \cref{Le: SemiStablePosDegNoGlobal}. If $0\leqslant \mu(\E)\leqslant 2g-2$, then Clifford's theorem for semistable vector bundles~\cite[Theorem 2.1]{Brambila-Paz-Grzegorczyk-Newstead1997Brill-Noether-loci} gives 
    \[
    h^0(\C,\E) \leqslant \rk(\E)(1+\mu(\E)/2)\leqslant \rk(\E)g.
    \]
    Finally, if $\mu(\E)>2g-2$, 
    then by Serre duality 
    and
    \cref{Le: SemiStablePosDegNoGlobal}
    \[
    h^1(\C,\E)= h^0(\C, \E^\vee\otimes \omega_{\C})=0,
    \]
    since $\E^\vee\otimes \omega_{\C}$ is semistable of slope $-\mu(\E)+2g-2<0$. 
    In that case 
    \[
    h^0 ( \C , \E ) 
    = 
    \rk(\E) ( \mu(\E) + 1 - g ).
    \] 
    Combining these three cases now easily yields the result.  

    To show that $\mathcal O_\C \otimes H^0 ( \C , \E ) \longrightarrow \E$ is surjective, 
    we can assume that $k$ is algebraically closed and
    take $x$ to be a point of $\C$ and $\E_x$ be the fibre of $\E$ at $x$.
    The exact sequence $0\to \E ( - x ) \to \E \to \E_x \to 0$ gives us a short exact sequence for global sections (since $h^1 ( \E ( - x ) ) = 0$), in particular a surjection from $H^0 ( \C , \E )$ to the fibre of $\E$ at $x$. From there Nakayama's lemma allows one to deduce that $\mathcal O_\C \otimes H^0 ( \C , \E ) \longrightarrow \E$ as vector bundles. 
\end{proof}

\begin{lemma}\label{Le: Upper.Bound.h0}
    We have 
    \[
    h^0(\C,\E)\leqslant \sum_{i=1}^r\max\{\rk(\E_i/\E_{i-1})(\lambda_i+1-g), \rk(\E_i/\E_{i-1})g\}
    \]
    and if $h^1(\C,\E)> 0$, then 
    \[
    h^0(\C,\E)\leqslant \sum_{i=1}^{r-1}\max\{\rk(\E_i/\E_{i-1})(\lambda_i+1-g), \rk(\E_i/\E_{i-1})g\}+2g-2.
    \]
\end{lemma}
\begin{proof}
        The short exact sequence 
    \[
    \begin{tikzcd}[sep=small]
        0 \arrow{r} 
        & \E_{i-1}\arrow{r} 
        & \E_i \arrow{r} 
        & \E_i/\E_{i-1}\arrow{r} 
        & 0
    \end{tikzcd}
    \]
    for $1\leqslant i \leqslant r$ together with the long exact sequence in cohomology implies 
    \begin{align}
    \begin{split}\label{Eq: LALALALA}
        h^0(\C,\E) 
        & \leqslant h^0 \left ( \C, \E_{r-1} \right ) + h^0 \left (\C, \E_r/\E_{r-1} \right ) \\
        & \leqslant h^0\left (\C, \E_{r-2}\right )+h^0 \left (\C, \E_{r-1}/\E_{r-2}\right ) +h^0\left (\C, \E_r/\E_{r-1}\right ) \\
        & \phantom{\leq}\vdots \\
        &\leqslant 
        \sum_{i=1}^r 
            h^0 \left ( \C, \E_i/\E_{i-1} \right ). 
        \end{split}
    \end{align}
    The first statement of the result now follows from Lemma~\ref{Le: Upperbound.h^0.semistable} and the second from the first combined with Lemma~\ref{Le: h^1not0implieslambda_rsmall}.
\end{proof}


\subsubsection{Algebraic families of vector bundles}
We will have to deal with families of vector bundles in the following sense. 

\begin{definition}
    Let $S$ be a scheme over $k$. An 
    \emph{algebraic family of vector bundles parameterised by $S$}
    is a vector bundle on $X \times S$ that is \emph{flat over $S$}. 
\end{definition}
We will also need the fact that, loosely speaking, slopes induce constructible partitions in the moduli space of vector bundles. 

\begin{proposition}[{\cite[Proposition 10]{shatz1977decomposition}}]
    Assume that $X$ is an irreducible, smooth, projective variety
    and that $S$ is a locally Noetherian scheme over $k$. 
    Let $\mathcal E$ be an algebraic family of vector bundles on $X$ parameterised by $S$.

    Then the function sending $s \in S$ to the Harder--Narasimhan polygon of $\mathcal E_s$ is a constructible function on $S$. 
\end{proposition}


\subsubsection{Vector bundles from the adelic point of view} \label{SubSe: VecBundAdele}
Of special importance to us are vector bundles defined via local data and so we will quickly recall their construction. 

Let $r\geq 1$ be an integer. Suppose that for every closed point $p$ of $\C$ we are given a free $\widehat{\OK}_{\C,p}$ module $M_p\subset \widehat{K}_p^r$ of rank $r$ such that $M_p=\widehat{\OK}_{\C,p}^r$ for all but finitely many $p$. 
Then we can associate to this data a vector bundle $\E$ of rank $r$ over $\C$ as follows. Given $U\subset \C$ open, set 
\[
\E(U)= \{ 
\bm f \in k(\C)^r
    \mid 
\bm f \in M_p \text{ for all }
p\in U\}.
\]
It is easy to check that the assignment $U\mapsto \E(U)$ defines a locally free sheaf of rank $r$ over $\C$ such that the completion of its stalks are given by $\E_p=M_p$. In addition, by definition we have 
\[
H^0(\C,\E)
= 
\{
    \bm f \in k(\C)^r 
        \mid 
    \bm f \in M_p\text{ for all }
        p \in |\C|
\} . 
\]

Now we perform this construction in families. 
Let $m$ be a positive integer and $\varpi = ( m_i )_{i\in \mathbf N}$ a generalised partition of $m$, that is to say $\sum_i i m_i = m$. 
Recall that the $k$-scheme $\Sym_{/k}^\varpi ( \mathscr C )_*$ 
parameterises effective divisors 
on $\mathscr C$ of the form 
\[
\sum_{i\in \mathbf N} i (  [p_{i1}] + ... + [p_{im_i}] )
\]
(as geometric points)
where the $p_{ij}$ are pairwise distinct.

For every tuple $\bm D = ( D_1 , ... ,D_r )$ of effective divisors parameterised by $ \Sym_{/k}^\varpi ( \mathscr C )_*$
and any closed point $p$ 
of $\mathscr C ' = \mathscr C \times_k \kappa ( \bm D )$,
let $M_p ( \bm D )$ be the free $\widehat{\OK_{\mathscr C' , p }}$-module of rank $r$ given by
\[
M_p ( \bm D ) = \bigoplus_{i=1}^r \mathfrak{m}_p^{v_p(D_i)}\widehat{\OK_{\C',p}} . 
\]

Let $S = \Sym_{/k}^\varpi ( \mathscr C )_*^r$ and  let $\mathcal E$ be the vector bundle 
on $\mathscr C \times S$
given by 
\begin{align*}
& \mathcal E ( U / S ) \\
& =
\left \{  
    ( \bm f , \bm D ) 
    \in k ( \mathscr C )^r \times \pr_2 ( U ) 
    \mid 
    \bm f \in M_p ( \bm D ) 
    \text{ for all } ( p , \bm D ) \in U 
\right \} 
\end{align*}
for every open 
$U \subset \mathscr C \times S$. 
By generic flatness, we have the following. 
\begin{lemma}
There exists a stratification of $S = \Sym_{/k}^\varpi ( \mathscr C )_*^r$ such that on each stratum, the sheaf $\mathcal E $ defines an algebraic family of vector bundles. 
\end{lemma}


\section{Motivic preliminaries}\label{Se: MotivicPrelim}

\subsection{Rings of varieties and motivic measures}

A general reference for this section is~\cite{chambert-loir-nicaise-sebag2018motivic}. 

\subsubsection{Definitions}
Let $S$ be a scheme. An $S$-variety is a morphism of schemes $X\to S$ of finite presentation.

\begin{definition}
    The Grothendieck group 
    $\KVar S$
    of varieties over $S$ 
    is the free $\mathbf Z$-module generated by isomorphism classes of $S$-varieties $X\to S$
    modulo \emph{cut and paste} relations
    \[
    [ X ] - [ Z ] - [ X - Z ]
    \]
    whenever $Z$ is a closed $S$-subscheme of an $S$-variety $X$.
    The Cartesian product endows $\KVar S$
    with a ring structure given by 
    \[
    [ X ] [ Y ]
        =
    [ X \times_S Y]. 
    \]
\end{definition}

In positive characteristic, we will have to work with the following variant
of $\KVar S$. 

\begin{definition}
    The ring 
    $\KVar S^\mathrm{uh}$
    is the variant of $\KVar S$
    obtained by
    identifying classes of varieties that are universally homeomorphic. 
    When $S = \Spec ( k )$, this corresponds to radicial surjective morphisms. 
\end{definition}

\begin{remark}
    In characteristic zero, this modification has no effect: the quotient morphism
    \[ 
    \KVar k \to 
    \KVar k^\mathrm{uh} 
    \]
    is an isomorphism, 
    see for example~\cite[Corollary 4.4.7]{chambert-loir-nicaise-sebag2018motivic}. 
\end{remark}

\begin{notation}
Let $\mathbf L_S$ be the class of the affine line in $\KVar S^\mathrm{uh}$.
We write
\[
\mathscr M_S 
= 
\KVar S^\mathrm{uh} [ \mathbf L_S^{-1}] .
\]
\end{notation}

The following important lemma allows us to check equalities of classes on the level of points
and we will use it several times without necessarily referring directly to it. 

\begin{lemma}[{\cite[Lemma 1.1.8]{chambert-loir-loeser2016motivic}
	or 
	\cite[Theorem 1]{cluckers-halupczok2022evaluation}}]
 \label{lemma-motivic-functions-are-defined-on-points}
	Let $a \in \mathscr M_S $.
	If $s^* a = 0 $ 
	in $\mathscr M_{\kappa ( s )}$
	for all points $s \in S$, then $a = 0$ in $\mathscr M_S$.  
\end{lemma}

The localised ring 
$\mathscr M_S $
comes with a natural (decreasing) filtration by the virtual dimension, whose $m$-th part is the subgroup generated by the set 
\[
\{ 
[ X ] \mathbf L^{-i} 
\mid 
\dim ( X ) - i \leqslant - m
\} .
\]
Its completion with respect to this filtration is denoted by $\widehat{\mathscr M_S }^{\dim}$.

\subsubsection{Point-counting and motivic measures}

The ring of varieties comes with a number of ring morphisms
$\phi : \KVar S \to A$
called \emph{motivic measures}.
When $S = \Spec ( \mathbf F_q )$ and $A = \mathbf Z$,
counting $\mathbf F_q$-points
\[
\begin{array}{rll}
\#_{\mathbf F_q}  :    \KVar{\mathbf F_q}^\mathrm{(uh)} & \longrightarrow & \mathbf Z  \\
     {[ X ]} & \longmapsto &  \# X ( \mathbf F_q )  
\end{array}
\]
defines such a motivic measure,
which extends naturally to 
\[
\#_{\mathbf F_q}  :    \mathscr M_{\mathbf F_q}  \longrightarrow \mathbf Q . 
\]
\begin{remark}
    The $\mathbf F_q$-point-counting motivic measure is not continuous with respect to the dimensional topology on $\mathscr M_{\mathbf F_q}$. 
    Hence in general it is not possible to deduce the function field Manin--Peyre principle in positive characteristic 
    from the (dimensional) motivic Manin--Peyre. This issue has motivated recent works such as the Hadamard topology introduced by Bilu--Das--Howe~\cite{bilu-das-howe2022zeta-stat}. This is not a real problem for the current paper, since we reduce our counting problems to estimating the number of points on affine spaces, 
    but
    one can hope for a result unifying \cref{Thm:TheTheorem.Fq(C)} and \cref{Th: TheTheorem.Motivic} when $k = \mathbf F_q$. 
\end{remark}


\subsection{Toolbox from motivic arithmetics}

We collect a few definitions and properties of motivic analogues of arithmetic functions and Euler products. 

\subsubsection{Motivic Euler products}

We will heavily use the notion of \emph{motivic Euler product} introduced by Margaret Bilu in~\cite[Chap.3]{bilu2023MAMS}.
Such ``products'' are formal power series whose coefficients are defined in terms symmetric products of varieties and \emph{behave like products}. 
Thanks to this notion, one is able to make sense of expressions of the form 
\begin{equation}
\label{notation-motivic-Euler-product}
\prod_{y \in Y/S}
\left ( 
    1 + 
    \sum_{i\in I} a_i T_i 
\right )
\end{equation}
in $\KVar{S} [[ T_i ]]$,
where $Y$ is a variety above another variety $S$ and the $a_i$ are classes in $\KVar{Y}$
indexed by a set $I$. 
More explicitly, 
\eqref{notation-motivic-Euler-product}
is \emph{defined} as a notation for the formal power series
\[
\sum_{\varpi \in \mathbf N^{(I)}}
\Sym_{Y/S}^\varpi  
    \left ( 
    \left ( 
        a_i
    \right )_{i\in I}
    \right )_*
    \mathbf T^\varpi 
\]
where $\mathbf T^\varpi $ stands for $\prod_{i\in I} T_i^{n_i}$
whenever $\varpi = ( n_i )_{i\in I}$. 
Informally speaking, the coefficients of this series are exactly what one should obtain when trying to expand a product such as \eqref{notation-motivic-Euler-product}. 
The $*$ symbol is there to recall that above a given point $y \in Y$, we 
pick at most one of the $a_i$ 
(we restrict to the complement of a big diagonal).

Let us briefly explain how these coefficients are constructed. 
The first step is classical: for any $n \in \mathbf Z_{>0}$,
one starts from 
\[
\Sym^n_{/S} ( Y ) = ( Y \times_S ... \times_S Y ) / \mathfrak S_n
\]
and defines $\Sym^n_{/S} ( Y )_*$ to be the image 
by $ Y \times_S ... \times_S Y \to \Sym^n_{/S} ( Y ) $
of the complement of the big diagonal. 
Now given a $Y$-variety $A \to Y$, 
the $\Sym^n_{/S} ( Y )$-variety
\[
\Sym^n_{Y/S} ( A ) \to \Sym^n_{/S} ( Y )
\]
is defined the same way, 
and more generally 
\[
\Sym^\varpi_{Y/S}
    \left ( ( A_i )_{i\in I} \right )
=
\prod_{i\in I} \Sym^{n_i}_{Y/S} ( A_i ) . 
\]
Again, the restricted product $\Sym^{n_i}_{Y/S} \left ( ( A_i )_{i\in I} \right ) )_*$ is obtained by removing the big diagonal. 
Eventually this construction is extended to any family of classes above $Y$. 

The following multiplicative property is of great use to us.

\begin{proposition}[{\cite[Proposition 3.9.2.4]{bilu2023MAMS}}]
    \label{proposition-multiplicativity-of-Euler-products}
    Let $Y$ and $S$ be as above. Let $(a_i)$ and $(b_i)$
    be families of $Y$-varieties indexed by $I = I_0 - \{ 0 \}$,
    where $I_0$ is a commutative monoid. 

    Then 
    \[
\prod_{y \in Y/S}
    \left ( 
    \left ( 
    1 + 
    \sum_{i\in I} a_i T_i 
    \right )
    \left ( 
    1 + 
    \sum_{i\in I} b_i T_i 
    \right )
    \right )
    =
    \left ( 
    \prod_{y \in Y/S}
\left ( 
    1 + 
    \sum_{i\in I} a_i T_i 
\right )
\right )
\left ( 
\prod_{y \in Y/S}
\left ( 
    1 + 
    \sum_{i\in I} b_i T_i 
\right )
\right )
    \]
\end{proposition}


\subsubsection{Classical and motivic Möbius functions}
\label{subsection-Möbius-function}

When the curve $\mathscr C$ 
is defined over a finite field $\mathbf F_q$,
the classical Hasse--Weil zeta function of $\mathscr C$
\[
\zeta^{\mathrm{{HW}}}_\mathscr C ( q^{-s} )
=
\prod_{x \in | \mathscr C |}
\left ( 
1 - q^{s \deg ( x )}
\right )^{-1}
\]
naturally comes with its inverse whose coefficients encode the Möbius function 
\[
\prod_{p \in | \mathscr C |}
\left ( 
1 - q^{s \deg ( p )}
\right )
= 
\sum_{D \in \DDiv ( \mathscr C )} \mu_{\mathscr C} ( D ) q^{s \deg ( D )} . 
\]
Indeed, the formula 
\[
\sum_{E \leqslant D}
\mu_{\mathscr C} ( E )  
= 
\begin{cases}
    1 & \text{ if } D = 0 \\
    0 & \text{ otherwise}
\end{cases}
\]
uniquely determines $\mu_\mathscr C$. 
Equivalently, one can expand the product to obtain the usual definition of $\mu_\mathscr C$ as 
\[
\mu_{\mathscr C} ( D ) 
= 
\begin{cases}
    ( - 1 )^{s} & \text{ if $D = p_1 + ... + p_s$  with $p_i$ pairwise distinct,} \\
    0 & \text{ otherwise.}
\end{cases}
\]

In a similar fashion, when $\mathscr C$ is defined over an arbitrary field $k$,
the Kapranov zeta function of $\mathscr C$
\[
Z_\mathscr C^{\Kapr}
( T ) 
= 
\prod_{p \in \mathscr C}
( 1 - T )^{-1}
=
\sum_{d \in \mathbf N}
\left [
\DDiv^d ( \mathscr C )
\right ] T^d 
\]
admits an inverse
\[
\prod_{p \in \mathscr C}
( 1 - T )
=
\sum_{d \in \mathbf N}
\mu^{\mathrm{mot}}_\mathscr C ( d ) T^d 
\]
where for any $d \in \mathbf N$
the term 
$\mu^{\mathrm{mot}}_\mathscr C ( d ) $
is a motivic function on $\DDiv^d ( \mathscr C )$,
that is, an element of $\KVar{\DDiv^d ( \mathscr C )}^\mathrm{uh}$,
and $\mu_\mathscr C^{\mathrm{mot}}$
is a motivic function on $\DDiv ( \mathscr C )$
satisfying 
\[
\mu_\mathscr C^{\mathrm{mot}} ( D ) 
\mu_\mathscr C^{\mathrm{mot}} ( D ' )
=
\mu_\mathscr C^{\mathrm{mot}} ( D + D ' )
\]
whenever $D$ and $D'$ have disjoint supports. 
When $k = \mathbf F_q$,
these definitions specialise to the classical ones via the counting measure.

These naive definitions 
can already be used to remove coprimality conditions 
between two divisors. For example, 
the coefficients of the motivic Euler product 
\[
\prod_{p \in \mathscr C}
\left ( 
1 + \sum_{m_1 \geqslant 1} T_1^{m_1}
+ 
\sum_{m_2 \geqslant 1} T_2^{m_2} 
\right )
\]
encodes $2$-tuples of effective divisors having disjoint supports. But 
thanks to \cref{proposition-multiplicativity-of-Euler-products},
we have the relation 
\begin{align*}
\prod_{p \in \mathscr C}
\left ( 
1 + \sum_{m_1 \geqslant 1} T_1^{m_1}
+ 
\sum_{m_2 \geqslant 1} T_2^{m_2} 
\right )
& = 
\prod_{p \in \mathscr C}
\left ( 
\left ( 
1 - T_1 T_2 
\right ) \sum_{( m_1 , m_2 ) \in \mathbf N^2} T_1^{m_1} T_2^{m_2}  
\right )
 \\
& =
Z_\mathscr C^{\Kapr}
( T_1 ) 
Z_\mathscr C^{\Kapr}
( T_2 ) 
\sum_{d \in \mathbf N}
\mu^{\mathrm{mot}}_\mathscr C ( d ) (T_1 T_2 )^d  . 
\end{align*}
It is important to note that this relation already holds above the scheme $\DDiv ( \mathscr C )^2$. 
This can be done for a higher number of variables:
\begin{align*}
\prod_{p \in \mathscr C}
\left ( 
\sum_{
\substack{
\bm m \in \mathbf N^r \\
( m_i , m_j ) \ngeqslant ( 1 , 1 ) \\
\forall i \neq j  
}
} 
\mathbf T^{\bm m}
\right )
& = 
\prod_{p \in \mathscr C}
\left ( 
( 1 - T_k T_l )
\sum_{
\substack{
\bm m \in \mathbf N^r \\
( m_i , m_j ) \ngeqslant ( 1 , 1 ) \\
\forall \{ i \neq j \} \neq \{ k \neq l \} 
}
} \mathbf T^{\bm m} 
\right )
\\
& 
= \prod_{p \in \mathscr C}
\left ( 
\prod_{k\neq l}
( 1 - T_k T_l )
\sum_{
\substack{
\bm m \in \mathbf N^r 
}
} \mathbf T^{\bm m} 
\right ) \\
& = 
\prod_{i=1}^r Z_{\mathscr C}^{\Kapr}
( T_i ) \prod_{k \neq l } \sum_{d_{kl} \in \mathbf N} 
\mu_\mathscr C^{\mathrm{mot}} ( d_{kl} ) ( T_k T_l )^{d_{kl} }.  
\end{align*}


\subsubsection{From coefficients to residues}

First we reproduce \cite[Lemme 12]{bourqui2011Manin-geometrique-quadriques-intrinseque}. 

\begin{lemma}
    \label{lemma-main-term-general-error-estimate-Fq}
    Let $n\geqslant 1$ be an integer and $q >0$ a real number. Let 
    \[
    F ( \mathbf T ) = \sum_{\bm d \in \mathbf N^n} a_{\bm d} \mathbf T^{\bm d}
    \]
    be a power series with complex coefficients. 

    Assume that $F ( \mathbf T )$ converges absolutely on the polydisk of radius $(r, ..., r)$ with $r>1/q$. 
    Let 
    \[
    \sum_{\bm d \in \mathbf N^n} b_{\bm d} \mathbf T^{\bm d} 
    =
    \frac{F ( \mathbf T )}{\prod_{i=1}^n ( 1 - q T_i )}. 
    \]
    Then for any $\varepsilon >0 $ small enough,
    \[
    \left |
    b_{\bm d} - F ( q^{-1} , ... , q^{-1} ) q^{| \bm d |}
    \right |
    \leqslant 
    \frac{q^{-\varepsilon} \Vert F \Vert_{q^{-1+\varepsilon}}}{( 1 - q^{-\epsilon} )^n}
        \sum_{i=1}^n q^{
            | \bm d | - \varepsilon d_i 
                } 
    \]
    for all $\bm d \in \mathbf N^n$,
    where $\Vert F \Vert_\eta = \max_{|T_i|=\eta} | F ( \mathbf T ) |$. 

    In particular,
    \[
     \left |
    b_{\bm d}q^{-| \bm d |} - F ( q^{-1} , ... , q^{-1} ) 
    \right |
    \ll_\varepsilon 
    q^{ - \varepsilon \min ( d_i ) } .
    \]
\end{lemma}

An analysis of the proof of~\cite[Lemme 12]{bourqui2011Manin-geometrique-quadriques-intrinseque}
shows that it can be easily adapted to the dimensional filtration as follows. 

\begin{lemma}
\label{lemma-main-term-general-error-estimate-DIM}
Let $S$ be a scheme and
$( \mathfrak a_{\bm d} )_{\bm d \in \mathbf N^n}$
be a family of classes in $\mathscr M_S$.
Consider the series
    \[
F ( \mathbf T )
= 
\sum_{\bm d \in \mathbf N^n}
\mathfrak a_{\bm d}
\mathbf T^{\bm d}.
    \]
Let $\eta >0$. 
Assume that $F ( \mathbf T )$ converges absolutely 
in $\widehat{\mathscr M_S}^{\dim}$
for $| T_i | < \LL^{-1 + \eta}$
and 
consider the family $( \mathfrak b_{\bm d} )_{\bm d \in \mathbf N^n}$ 
\[
\frac{F ( \mathbf T ) }{
\prod_{i=1}^n ( 1 - \LL T_i )
}
= 
\sum_{\bm d \in \mathbf N^n} 
    \mathfrak b_{\bm d} \mathbf T^{\bm d} .
\]
Then for all $ 0 < \varepsilon < \eta$, the error term 
\[
\mathfrak b_{\bm d} \LL^{ - | \bm d | }
- 
F ( \LL^{- 1 } ) \in \widehat{\mathscr M_S}^{\dim}
\]
has virtual dimension bounded above by 
$
- \varepsilon \min ( d_i ) 
$ plus a constant that does not depend on $\bm d$. 

\end{lemma}


\section{Lifting to the universal torsor and a Möbius inversion}\label{Se: LiftingUT}

In what follows, it will be convenient to adopt the following notation. 
\begin{definition}
Given divisors $D, E$ on $\C$, their ``greatest common divisor'' and ``least common multiple'' are 
\begin{align*}
    (D;E) & = \sum_{P\in |\C|}\min\{v_P(D),v_P(E)\}P,\\
    [D;E] & = \sum_{P \in |\C|}\max \{v_P(D),v_P(E)\} P. 
\end{align*}
\end{definition}

Recall that we assume that the $k$-curve $\mathscr C$
admits a degree one $k$-divisor
$\mathfrak D_1$. 
This is automatically the case for example when $k = \mathbf F_q$
or when $k$ is algebraically closed. 
\begin{notation}
If $\E$ is a vector bundle on $\C$ and $d\in \mathbf{Z}$, then we shall write $\E(d)$ for $\E\otimes \OK_\C(d\mathfrak{D}_1)$. 
\end{notation}


\subsection{Passing to the universal torsor}
Let $X$ be a split del Pezzo surface of degree $5$ over $k$.
Given $\bm{d}\in \CEff(X)^\vee\cap \Pic(X)^\vee\subset \mathbf{N}^{10}$, we will now recall the parameterisation of the space of morphisms $\Hom^{\bm{d}}_k(\C,X)_U$ in terms of the universal torsor of $X$. 

Let us define
\[
\mathfrak I
=
\{ 
1 , 2 , 3 , 4 , 12 , 23 , 34 , 41 , 13 , 24 
\}, 
\]
with the convention $ij=ji$ for every $i\neq j$, and 
\[
I = 
    \{ 1 , 2 , 3 , 4 \}, 
\quad 
J =
    \{ 12 , 23 , 34 , 41 , 13 , 24 \}.
\]

\begin{notation}\label{notation - space of sections}
Given a subset $K\subset \mathfrak{I}$ and $\underline L \in \PPic^0 ( \C )^K$, define 
\[
\mathcal{H}_{\bm{d}_K}
( \underline L )
=
\prod_{k\in K}
    H^0(\C, L_k(d_k))
\]
and 
\[
\mathcal{H}^\bullet_{\bm{d}_K}
( \underline L )
=
\prod_{k\in K}
    \left(H^0(\C, L_k(d_k))\setminus\{0\}\right) .
\]
Furthermore, 
we denote by $\mathcal{H}^\ast_{\bm{d}_K}$ the Zariski-open subset of  
$\mathcal{H}^\bullet_{\bm{d}_K}$ parameterising tuples of non-zero sections whose associated divisors have pairwise disjoint support. 
For the sake of brevity, we shall write $\mathcal{H}^\bullet_{\bm d}$ and $\mathcal{H}_{\bm d}$ for $\mathcal{H}^\bullet_{\bm{d}_\mathfrak{I}}$ and $\mathcal{H}_{\bm{d}_\mathfrak{I}}$ respectively. Similarly, we shall write $\mathcal{H}_{\bm{d}'}$, $\mathcal{H}_{\bm{d}'}^\bullet$, $\mathcal{H}_{\bm{d}''}$, $\mathcal{H}_{\bm{d}''}^\bullet$ instead of $\mathcal{H}_{\bm{d}_I}$, $\mathcal{H}_{\bm{d}_I}^\bullet$, $\mathcal{H}_{\bm{d}_J}$, $\mathcal{H}_{\bm{d}_J}^\bullet$ respectively. Finally, we shall omit the dependence on $\underline L$ whenever it is clear from the context and write $\underline{L}'$, $\underline{L}''$ for $\underline{L}_I$, $\underline{L}_J$ respectively. 
\end{notation}

\begin{definition}\label{Defi: TheCountingSpace}
    For any $\bm d \in \CEff ( X )^\vee \cap \Pic ( X )^\vee
    \hookrightarrow \mathbf N^{10}
    $ and $\underline L\in \PPic^0(\C)^{10}$, 
    let $M ( \bm d ,\underline L)$ be the parameter space of tuples 
    \[
    ( a_1 , a_2 , a_3 , a_4 , a_{12} , a_{23} , a_{34} , a_{13} , a_{24} , a_{14} )\in \mathcal{H}^\bullet_{\bm{d}} ( \underline L ) 
    \]
    satisfying the torsor relations 
    \begin{equation}
    \label{equation-pluckers-relations}
    \begin{array}{rcrcrcr}
        a_4a_{14}&-&a_3a_{13}&+&a_2a_{12}&=&0,\\
        a_4a_{24}&-&a_3a_{23}&+&a_1a_{12}&=&0,\\
        a_4a_{34}&-&a_2a_{23}&+&a_1a_{13}&=&0,\\
        a_3a_{34}&-&a_2a_{24}&+&a_1a_{14}&=&0,\\
        a_{12}a_{34}&-&a_{13}a_{24}&+&a_{23}a_{14}&=&0,
    \end{array}
\tag{P}
\end{equation}
and the coprimality conditions 
\begin{equation}
    \label{coprimality-condition-**}
    \begin{aligned}
        (\ddiv(a_i);\ddiv(a_j))&=0, \\
        (\ddiv(a_i);\ddiv(a_{jk}))&=0 ,\\
        (\ddiv(a_{ij});\ddiv(a_{jk}))&=0.
    \end{aligned}
    \tag{$\ast \ast$}
\end{equation}
\end{definition}
Recall that there is a short exact sequence
\begin{equation}
    \label{exact-sequence-relation-between-lines}
    \begin{tikzcd}
        0 \rar & N_X = k[U]^\times / k^\times \rar & \oplus_{i\in \mathfrak I} D_i \rar & \Pic ( X ) \rar & 0 
    \end{tikzcd}
\end{equation}
encoding the linear relations between the lines, where $D_i$ runs the lines of $X$ and $U\subset X$ is the complement of lines. 

An application of \cite[Corollary 6.8]{faisant2025univ-torsors}
provides the following proposition; see also~\cite[Example 5.3]{faisant2025univ-torsors}. 

\begin{proposition}\label{Prop: Passing.UT}
     For any $\bm d \in \CEff ( X )^\vee \cap \Pic ( X )^\vee
    \hookrightarrow \mathbf N^{10}
    $, 
    we have 
   \[
   \sum_{
        \substack{
        \underline L\in \PPic^0(\C)^{10}
        \\
        \otimes_{i\in \mathfrak{I}} L_i^{\otimes n_i} \simeq \mathcal O_\mathscr C 
        \forall \bm n \in N_X
        }
    }
   [ M ( \bm d , \underline L ) ]
   = 
   ( \mathbf L_k - 1 )^5 
   \left [
        \Hom_k^{\bm d} ( \mathscr C , X )_U 
   \right ] 
   \]
   in $\KVar k$. 
\end{proposition}
The space defined by \eqref{equation-pluckers-relations} and $a_i, a_{jk}\neq 0$ is the \emph{universal torsor} of $X$ and comes with an action by the N\'eron--Severi Torus of $X$, which induces an action of $k(\C)^{\times, 5}$ on the $k(\C)$-points of the universal torsor. Explicitly, for $\bm{u}=(u_0,\dots, u_4)\in k(\C)^{\times, 5}$ the action is given by 
\begin{equation}
\label{Eq: Action.NSTorus}
\bm{u} \cdot ( a_1,\dots, a_4, a_{12}, \dots, a_{34} ) 
= 
(u_1a_1,\dots, u_4a_4, u_0u_{1}^{-1}u_2^{-1}a_{12},\dots, u_0u_{3}^{-1}u_4^{-1}a_{34}),
\end{equation}
corresponding to the $\Pic(X)$ grading of the Cox ring of $X$. 

In addition, if $E_1,\dots, E_4$ are pairwise disjoint lines, then there is a realisation of $X$ as a blow-up of $\PP^2$ having $E_1,\dots, E_4$ as exceptional divisors. This corresponds to an automorphism of the universal torsor of $X$ that sends $a_i$ to the variable corresponding to $E_i$ and $a_{ij}$ to the variable corresponding to the strict transforms of the lines through the blown-up points $P_i, P_j\in \PP^2$ above which $E_i$ and $E_j$ are the exceptional divisors.  

We will make critical use of both the action by the N\'eron--Severi torus and the automorphisms of the universal torsor in our arguments.

It will be convenient to translate the Pl\"ucker relations into congruence conditions, as recorded in the next lemma. 
\begin{lemma}\label{Le: CongruenceImpliesPluecker}
    Let $\bm{d}\in \CEff(X)^\vee\cap \Pic(X)^\vee$ and $\underline L\in \PPic^0(\C)^{10}$ satisfying $\otimes_i L^{n_i} \simeq \mathcal O_\mathscr C $ for all $ \bm n \in N_X$. 
    If $\bm{a}'\in \mathcal{H}^\bullet_{\bm{d}'}$ and $a_{ij}\in H^0(\C, L_{ij}(d_{ij}))$ for 
    $ij\in \{13, 24, 34\}$ satisfy 
    \begin{equation}
        \begin{aligned}
            \ddiv(a_1) & \leqslant \ddiv(a_3a_{34}-a_2a_{24}),\\
            \ddiv(a_2)&\leqslant \ddiv(a_4a_{34}+a_1a_{13}),           
        \end{aligned}
    \end{equation}
    then $(\bm{a}',\bm{a}'')$ satisfies \eqref{equation-pluckers-relations}, where 
    \begin{align*}
        a_{14}&=\frac{a_2a_{24}-a_3a_{34}}{a_1}\in H^0(\C, L_{14}(d_{14})),\\
        a_{23}&=\frac{a_4a_{34}+a_1a_{13}}{a_2}\in H^0(\C,L_{23}(d_{23})),\\
        a_{12}&= \frac{a_3a_4a_{34}+a_1a_3a_{13}-a_2a_4a_{24}}{a_1a_2}.
    \end{align*}
    In addition, we have $a_{12}\in H^0(\C, L_{12}(d_{12}))$ if $(\ddiv(a_1);\ddiv(a_2))=0$. 
\end{lemma}
\begin{proof}
    It easily follows from a direct computation that the tuple $(\bm{a}',\bm{a}'')$ indeed satisfies \eqref{equation-pluckers-relations}. Since $\bm{d}\in \Pic(X)^\vee$, we have $d_{14}=d_{i}+d_{i4}-d_1$ for any $1\leqslant i \leqslant 4$. Similarly, we have $L_{14}=L_i\otimes L_{i4} \otimes L_{1}^{-1}$ from which $a_{14}\in H^0(\C, L_{14}(d_{14}))$ easily follows.
    A similar argument shows that $a_{23} \in H^0 ( \mathscr C , L_{23} ( d_{23} )$.
    Finally, the condition $(\ddiv(a_1); \ddiv(a_2))=0$ ensures that  $a_{12}$ indeed belongs to $H^0 ( \mathscr C , L_{12} ( d_{12} ))$. 
\end{proof}

\begin{remark}
    In the statement of the previous lemma, we picked an $\underline L$
    whose residue field $k'$ may not necessarily be $k$, hence $\mathscr C $ should be replaced by $\mathscr C_{k'} = \mathscr C \times_k \Spec ( k ' )$, which we did not do for the sake of readability. 
    It is clear that the proof above does not depend on $k$ nor $k'$. This situation will show up again several times in the rest of the paper. 
    In particular,
    we will not write $k'$ in place of $k$ whenever an $\underline L$
    or latter a tuple of divisors 
    is fixed. 
    
    Then, the uniformity of our computations with respect to the different algebraic parameters in play 
    will allow us to apply several times \cref{lemma-motivic-functions-are-defined-on-points} (which states that a motivic function is always determined by its values)  without explicitly referring to it.
\end{remark}


\subsection{A first upper bound}

Recall that given $\bm{d}\in \CEff(X)^\vee\cap\Pic(X)^\vee\subset \mathbf{N}^{10}$ and $\underline L\in \PPic^0(\C)^{10}$, we have defined the space $M(\bm{d}, \underline L)$ in \cref{Defi: TheCountingSpace}.  Our goal is now to asymptotically evaluate $M(\bm d, \underline L)$ under suitable assumptions on $\bm d$. 
Before commencing our analysis, we establish the following preliminary upper bound that will be helpful at various places of our arguments. 

\begin{proposition}\label{Prop: GeneralUpperBound}\label{Lemma: GeneralUpperBound-DIM}
    Let $\bm d \in \CEff(X)^\vee\cap \Pic(X)^\vee\subset {\bf N}^{10}$ and $\underline L\in \PPic^0(\C)^{10}$ be given. If $\bm d\in \mathbf{Z}^{10}_{\geq 0}$ satisfies $d_1\leqslant d_2\leqslant d_3\leqslant d_4$,
    then 
     \[
    \dim ( 
    \{
    ( \bm{a}', \bm{a}'' )
    \in \mathcal{H}^\bullet_{\bm{d}'} 
    \times 
    \mathcal{H}_{\bm{d}''} \mid
    a_{34}\neq 0 \text{ and }
    \eqref{equation-pluckers-relations}
    \}
    ) 
    \leqslant d + 4 g . 
    \]
    Moreover, if $k = \mathbf F_q$, then
    \[
    \#_{\mathbf F_q}
    \{
     ( \bm{a}', \bm{a}'' )
    \in \mathcal{H}^\bullet_{\bm{d}'} 
    \times 
    \mathcal{H}_{\bm{d}''} \mid
    a_{34}\neq 0 \text{ and }
    \eqref{equation-pluckers-relations}
    \}
     \ll d^3 q^d. 
    \]
\end{proposition}

\begin{remark}
    We are not assuming that $\bm{d}$ belongs to one of the the cones $\mathcal{T}_{(E_1,E_2,E_3,E_4)}$ from \cref{Prop:ConeDecomp} as we will later apply it in a situation where this is not the case.
\end{remark}

\begin{proof}
Let $K=\{13, 34, 24\}$. 
By \cref{Le: CongruenceImpliesPluecker} we have
that 
\[
  \{
    ( \bm{a}', \bm{a}'' )
    \in \mathcal{H}^\bullet_{\bm{d}'} 
    \times 
    \mathcal{H}_{\bm{d}''} \mid
    a_{34}\neq 0 \text{ and }
    \eqref{equation-pluckers-relations}
    \}
\]
is a subvariety of 
\begin{equation}
      \left \{
    ( \bm{a}' , a_{13} , a_{34} , a_{24} )
    \in \mathcal{H}^\bullet_{\bm{d}'} 
    \times 
    \mathcal{H}_{\bm{d}_K} \; \left| \;
    a_{34}\neq 0 \text{ and }
    \begin{array}{l}
    \ddiv(a_1)\leqslant \ddiv(a_3a_{34}-a_2a_{24})\\
    \ddiv(a_2)\leqslant \ddiv(a_4a_{34}+a_1a_{13})
    \end{array}
    \right.
   \right \} .
   \label{variety-first-estimate}  
\end{equation} 

Let us first fix all variables but $a_4$ and $a_3$. 
Since $a_{34}\neq 0$, the space of sections 
$a_{4}\in H^0(\C, L_4(d_4))$ 
such that 
${\ddiv(a_2) \leqslant \ddiv(a_4a_{34}+a_1a_{13})}$ is either empty or a translate of a vector space of dimension 
\[
h^0(\C, L_4(d_4-d_2)\otimes \OK_\C(-\ddiv(a_2)+(\ddiv(a_2);\ddiv(a_{34}))),
\]
which by \cref{Le: Upperbound.h^0.semistable} 
is at most 
\[
\max\{d_4-d_2+\deg(\ddiv(a_2);\ddiv(a_{34}))+1-g, g\} \leqslant d_4-d_2+\deg(\ddiv(a_2);\ddiv(a_{34}))+g,
\]
since $d_2\leqslant d_4$ by assumption. Similarly, if there is an $a_3\in H^0(\C, L_3(d_3))$ such that $d_1\mathfrak{D}_1+\ddiv((a_3a_{34}-a_2a_{23})/a_1)\geq 0$, then $d_1\leqslant d_3$ implies that $a_3$ belongs to a translate of a $k$-vector space of dimension at most $d_3-d_1+\deg(\ddiv(a_1);\ddiv(a_{34}))+g$. 

As $d_{ij}\geq 0$ implies that 
\[
h^0(\C, L_{ij}(d_{ij}))\leqslant d_{ij}+g
\]
by \cref{Le: Upperbound.h^0.semistable} for $ij=13, 24$,
it follows that 
the fibre of \eqref{variety-first-estimate}
above 
\[
(a_1 , a_2 , a_{34} ) \in \mathcal H^\bullet_{\{ 1 , 2 , 34 \}}
\]
is an affine space of dimension 
bounded above by
\begin{align*}
  &   4g+d_{13}+d_{24}+d_3+d_4-d_1-d_2 
\\ & + \deg(
\ddiv(a_1);\ddiv(a_{34}))+\deg(\ddiv(a_2);\ddiv(a_{34})) .
\end{align*}
Let us write $C_1$ and $C_2$
for $(\ddiv(a_1);\ddiv(a_{34}))$ and $((\ddiv(a_2);\ddiv(a_{34}))$ respectively.
This allows us to write the class of \eqref{variety-first-estimate}
as a power of $\mathbf L$
(whose exponent is bounded by the first line of the quantity above)
times 
\[
\sum_{
    \substack{C_1 , C_2 , C_{12}}
}
\mathbf 1_{(C_1 ; C_2 ) = C_{12}} 
        \sum_{
           \substack{a_1 \\ C_1 \leqslant \ddiv ( a_1 ) }
      }
      \sum_{
        \substack{a_2 \\ C_2 \leqslant \ddiv ( a_2 ) }
      }
      \sum_{
          \substack{a_{34} \\ C_1 + C_2 - C_{12} \leqslant \ddiv ( a_{34} ) }
      }
    \mathbf L^{\deg ( C_1 ) + \deg ( C_2 )} . 
\]
We can stratify
$\mathcal H^\bullet_{\bm d_{\{ 1 , 2 , 34 \}}}  $ 
with respect to the degrees of $C_1, C_2$ and $C_{12} = ( C_1 ; C_2 )$
and use that 
\[
\left \{
    ( a_1 , a_2 , a_{34} ) \in \mathcal H^\bullet_{\bm d_{\{ 1 , 2 , 34 \}}}
    \; \left  | \;
    \begin{array}{l}
    C_1 \leqslant \ddiv ( a_1 ),    C_2 \leqslant \ddiv ( a_2) \\
    C_1 + C_2 - C_{12} \leqslant \ddiv ( a_{34} )\\
    \end{array}
\right.
\right \} 
\]
is isomorphic to 
\[
(
\mathbf A^{d_1 - \deg ( C_1 )} \setminus \{ 0 \} ) 
\times 
(
\mathbf A^{d_2 - \deg ( C_2 )}
\setminus \{ 0 \} ) 
\times 
(
\mathbf A^{d_{34} - \deg ( C_1 ) - \deg ( C_2 ) + \deg ( C_{12} ) }
\setminus \{ 0 \} )  . 
\]
Since $\deg ( C_{12} ) \leqslant \deg ( C_1 ) + \deg ( C_2 ) $ we get the final dimensional upper bound 
\begin{align*}
  &   4g+d_{13}+d_{24}+d_3+d_4-d_1-d_2 
\\ & + d_1 + d_2 + d_{34} \\
& = 4g + d_{13}+d_{34}+d_{24}+d_3+d_4\\
& = 4g + d . 
\end{align*}

Applying the counting measure,
we get that the sum over the degrees is then 
at most 
\[
O_{q,g} (
    d_1 d_2 \min ( d_1 , d_2 )                               q^{d_{13}+d_{34}+d_{24}+d_3+d_4} 
    )
=
O_{q,g} (d^3 q^d) 
\]
as claimed. 
\end{proof}


\subsection{Möbius inversion}

Suppose that  
\[
{\bm a}' \in \mathcal H^*_{\bm{d}'},
\]
which means that the coprimality conditions
\[
( \ddiv ( a_i ) ; \ddiv ( a_j ) ) = 0 
\quad \forall \, i,j \in \{ 1 , 2 , 3 , 4 \}
\]
are satisfied. 

Recall that $M ( \bm d , \underline L )$ 
is the variety 
above $\Spec ( \kappa ( \underline L ) )$
parameterising tuples
\[
(\bm a ' , \bm a '' )
\in 
\mathcal H^\bullet_{\bm d} ( \underline L )
\]
satisfying the coprimality conditions $(**)$
as well as 
the Plückers relations \eqref{equation-pluckers-relations}. 
We denote by $M ( \bm d , \underline L ; \bm a '  )$
the class of its fibre above a given point $\bm a ' \in \mathcal H^\bullet_{\bm{d}'}$. 
In particular, we have 
\begin{equation}\label{Eq: Decomp.M(d,L).Mainbody}
[ ( \bm d  , \underline L ) 
\mapsto 
M(\bm{d}, \underline L) ] 
    = 
\left [ ( \bm d  , \underline L ) 
\mapsto 
    \sum_{
    \bm{a}'\in \
    \mathcal{H}^*_{\bm{d}'} }
    M(\bm{d}, \underline L ; \bm a ' ) 
\right ]
\end{equation}
in $\KVar{\PPic ( \mathscr C )^{10}}$. 

\begin{lemma}[Partial Möbius inversion]
\label{Le: M(a')=moebius.Inversion}
    The motivic function
    \[
    M ( \bm d  ) =
    \left [
    ( \bm a ' , \underline L '' ) 
    \mapsto 
    \{ 
    \bm a '' \in \mathcal H^\bullet_{\bm{d}''} ( \underline L'' )  
    \mid 
    ( \bm a ' , \bm a '' ) 
    \text{ satisfies } (**) 
    \text{ and } \eqref{equation-pluckers-relations}
    \}
    \right ]
    \in \KVar{\mathcal H^\bullet_{\bm{d}'}
    \times \PPic^0 ( \mathscr C )^{6}}
    \]
is given by 
\begin{align*}
& M ( \bm d  ; \bm a ' , \underline L)\\
& = 
\sum_{
    \substack{
    \bm D , \bm E \in \DDiv ( \mathscr C )^4
    \\
    (D_i , \ddiv ( a_j ) ) = 0 \; i \neq j 
    \\
    \bm E \leqslant \ddiv ( \bm a ' ) 
    }
}
\mu_\mathscr C ( \bm D, \bm E )
\left [
\bm{a}'' \in \mathcal{H}^\bullet_{\bm{d}''} ( \underline{L}'' )
\mid 
( \bm a ' , \bm a '' ) \text{ satisfies }
\eqref{equation-pluckers-relations}
\text{ and }
F_{ij} \leqslant \mathrm{div} ( a_{ij} )
\right ],
\end{align*}
where $\mu_{\C}$ was defined in \cref{subsection-Möbius-function}
and 
for every $\bm D , \bm E$,
$F_{ij} = [D_i ; D_j ] + E_k + E_l$.
\end{lemma}
\begin{proof}
   In the course of the proof, it will be convenient to adopt the convention that $i,j,k,l$ will always be such that $\{i,j,k,l\}=\{1,2,3,4\}$. 
   
   Let 
   us denote by 
   $\mathbf 1_{\eqref{coprimality-condition-**}}
   $
   the class of the open subset of $\DDiv ( \mathscr C )^{10}$
   parameterising tuples of effective divisors satisfying \eqref{coprimality-condition-**},
   seen as a variety above $\DDiv ( \mathscr C )^{10}$ via the embedding map. It should be thought of as an indicator function, hence the choice of notation.

   We first remove the coprimality condition between $A_{ij}$ and $A_{ik}$.
    Since $(A_{ij};A_l)=0$, for any divisors of sections satisfying \eqref{equation-pluckers-relations}, we have $( A_{ij} ; A_{ik} ) = ( A_{ij} ; A_{il} ) = ( A_{ik} ; A_{il} ) = 0 $
    if and only if $( A_{ij} ; A_{ik} ) = 0$. 
We obtain 
\[
\bm 1_{\eqref{coprimality-condition-**}} ( \bm A )
=
\sum_{
    \substack{
        \bm D \in \DDiv ( \mathscr C )^4 \\
        D_i \leqslant A_{ij} \forall i , j 
        }
    }
\mu_\mathscr C ( \bm D )  \mathbf 1_{(**')} ( \bm A ' , ( A_{ij} - D_i )_{i,j} )
\]
where $(**')$ is the new condition obtained from \eqref{coprimality-condition-**} by removing the coprimality condition between $A_{ij}$ and $A_{ik}$. Observe that the $i$th Plücker equation in \eqref{equation-pluckers-relations} also implies that $D_i\leq A_{il}+A_l$. As $(A_{ij};A_l)=0$ and $D_i\leq A_{ij}$, this gives $D_i\leq A_{il}$. Similarly, the condition $(A_{ij};A_k)=0$ also gives $(D_i;A_k)=0$. 

We can do that again for the coprimality condition between $A_i$ and $A_{jk}$
to obtain 
\begin{equation}
\label{eq:mobius-inversion-in-practice}
\bm 1_{\eqref{coprimality-condition-**}} ( \bm A )
=
\sum_{
    \substack{
        \bm D \in \DDiv ( \mathscr C )^4 \\
        D_i \leqslant A_{ij} \, \forall i \neq j \\
        ( D_i , A_j ) = 0 \\
       \bm E \in \DDiv ( \mathscr C )^4 \\
        E_i \leqslant A_i \, \forall i 
        }
    }
\mu_\mathscr C ( \bm D )  \mathbf 1_{(*)} ( \bm A ' )
\end{equation}
where the condition $(*)$ is just the coprimality condition between $A_i$ and $A_j$. 

Let us now see $M ( \bm d )$
as a motivic function on 
$\DDiv^{\bm d} ( \mathscr C )$. 
Let $H ( \bm d )$ be the motivic function 
on $\PPic ( \mathscr C )^{10}$ 
given by 
\[
H ( \bm d , \underline L )
= \mathcal H^\bullet_{\bm d} ( \underline L )
\]
and on $\DDiv ( \mathscr C )^{10}$ by 
\[
H ( \bm d , \bm A )
= 
\{ 
\bm a \in \mathcal H^\bullet_{\bm d} ( \underline L )
\mid 
\ddiv ( \bm a ) = \bm A
\}
\]
whenever $\bm A$ belongs to the tuple of linear classes given by $\underline L$. 
Using these notations we can rewrite $M ( \bm d )$ as 
\[
M ( \bm d )
=
\sum_{\bm A \in \DDiv^{\bm d} ( \mathscr C )}
\mathbf 1_{\eqref{coprimality-condition-**}}
( \bm A )
\left [
\bm a \in H ( \bm d , \bm A )
\mid \eqref{equation-pluckers-relations}
\right ] .
\]
Replacing $\mathbf 1_{\eqref{coprimality-condition-**}}$
by \eqref{eq:mobius-inversion-in-practice}
in the previous expression,
one obtains
\[
M ( \bm d , \bm A ' )
=
\sum_{
    \substack{
        \bm D \in \DDiv ( \mathscr C )^4 \\
        D_i \leqslant A_{ij} \, \forall i \neq j \\
        ( D_i , A_j ) = 0 \\
       \bm E \in \DDiv ( \mathscr C )^4 \\
        E_i \leqslant A_i \, \forall i 
        }
    }
\mu_\mathscr C ( \bm D )  
\left [
\bm a \in H ( \bm d , \bm A )
\mid \eqref{equation-pluckers-relations}
\right ]
\]
relatively to $\bm A' = \ddiv ( \bm a ' )$
satisfying the coprimality conditions $(A_i , A_j ) = 0$ for $i\neq j$.
Since the motivic function 
\[
A  \mapsto [ a  \mid \ddiv ( a  ) = A  ]
\]
is locally constant, taking values $0$ or $[\mathbf G_m ]$,
it also holds relatively to $\bm a ' \in \mathcal H^\bullet_{\bm d ' } ( \underline L ' )$ for any $\underline L '$.   
The statement of the result now follows upon observing that $D_i, D_j, E_k, E_l \leq A_{ij}$ is equivalent to
\[
[D_i;D_j]+E_k+E_l\leq A_{ij},
\]
since $(D_i;E_k)=(D_i;E_l)=(D_j;E_k)=(D_j;E_l)=(E_k;E_l)=0$, 
which by definition is to say that $F_{ij}\leq A_{ij}$. 
\end{proof}


\subsection{Truncation of the Möbius variables}

In this section we estimate the contribution of large M\"obius variables. The basic idea is that if one of them is large, we can use the action of the N\'eron--Severi torus to produce a morphism $\C\to X$ of smaller anticanonical degree. 
\begin{lemma}\label{Le: LargeMoebius.MainBody}
    Let $0<\delta_1 < 1/2$. 
    The virtual dimension
    of the motivic class
    \[
    \sum_{
        \bm{a}'\in \mathcal{H}_{\bm{d}'}^*}
    \sum_{
        \substack{
        \bm D , \bm E \in \DDiv( \C )^4 \\ \max_{i, j}(\deg(E_i), \deg(D_{j}))\geq \delta_1 d}}
    \left[ 
        \bm a ''\in \mathcal{H}_{\bm{d}''}^\bullet ( \underline L'')
        \; \left| \; 
            \begin{array}{l}
            ( \bm a ' , \bm a '' ) 
            \text{ satisfies \eqref{equation-pluckers-relations}}\\
            F_{ij}\leqslant \ddiv(a_{ij})
            \end{array}
            \right.
    \right ] 
    \] is bounded above by $d ( 1 - \delta_1 ) + 5g$. 

    Assume that $k = \mathbf F_q$. 
    Then for every $\varepsilon > 0$, 
    \begin{align*}
    & \#_{\mathbf F_q} 
    \left ( 
    \sum_{
        \bm{a}'\in \mathcal{H}_{\bm{d}'}^*}
    \sum_{
        \substack{
        \bm D , \bm E \in \DDiv( \C )^4 \\ \max_{i, j}(\deg(E_i), \deg(D_{j}))\geq \delta_1 d}}
    \left[ 
        \bm a ''\in \mathcal{H}_{\bm{d}''}^\bullet ( \underline L'')
        \; \left| \; 
            \begin{array}{l}
            ( \bm a ' , \bm a '' ) 
            \text{ satisfies \eqref{equation-pluckers-relations}}\\
            F_{ij}\leqslant \ddiv(a_{ij})
            \end{array}
            \right.
    \right ]
    \right ) \\
    & \ll_\varepsilon q^{d(1-\delta_1 +\varepsilon)} .
    \end{align*} 
\end{lemma}

\begin{remark}
    In the statement of the lemma, the summation is restricted to those $\bm{D}, \bm{E}\in \DDiv(\C)^4$ such that $(D_i;\ddiv(a_j))=0$ and $E_i \leq \ddiv(a_i)$ for every $i\neq j \in \{1,2,3,4\}$.  To keep the notation at a minimum, we will not write this condition explicitly in what follows. 
\end{remark}

\begin{proof}

We begin with a preliminary observation. Recall that the degree one divisor $\mathfrak{D}_1$ provides us with a splitting $\PPic(\C)\simeq \PPic^0(\C)\oplus \mathbf{Z}$. In particular, given a divisor $E\in \DDiv(\C)$, the line bundle $\OK_\C(E)$ is isomorphic to $L(\deg(E))$ for a unique $L\in \PPic^0(\C)$. We will write $\phi(E)$ for $L$, which is the image of $E$ under the Abel--Jacobi map under our identification $\PPic(\C)\simeq \PPic^0(\C)\oplus \mathbf{Z}$. 

Let us first consider the case where $\max_i \deg(E_i)\geq \delta_1 d$. 
Then given the symmetry at hand, we may assume that $\deg(E_1)\geq \delta_1 d$. 
We start with estimating the dimension 
of the parameter space 
\[
\left \{
(\bm{a}',\bm a'') \in \mathcal{H}^\bullet_{\bm{d}} \; \left | \;
\begin{array}{l}
    \eqref{equation-pluckers-relations}\text{ and }   E_1\leqslant \ddiv(a_1), \\
    E_1 \leqslant \ddiv(a_{jk}), \, 2\leqslant j, k\leqslant 4 
\end{array} \right.
\right \} . 
\]
Let $L=\phi(E_1)$ and $e=\deg(E_1)$. Pick a non-zero global section $u\in H^0(\C, L(e))$. Using the action of the Néron--Severi torus with 
\[
\bm u^{-1} = ( u^{-1} , u^{-1} , 1 , 1 ,1 )
\]
leads to a change of variables
\[
\bm b 
    = \bm u^{-1} \cdot ( \bm a', \bm a''),
\]
which satisfies \eqref{equation-pluckers-relations} and
\begin{align}
\begin{split}
    \label{NewBies}
    b_1 
    &
        \in H^0(\C, L_1\otimes L^{-1}(d_1-e)) \setminus\{ 0 \}, \\
    b_i
    &
        \in H^0(\C, L_i(d_i))\setminus\{ 0 \}, \text{ for }i=2,3,4,\\
    b_{ij}
    &
        \in H^0(\C,L_{ij}\otimes L^{-1}(d_{ij}-e))\setminus\{0\} \text{ for }2\leqslant i,j\leqslant 4,\\
    b_{1k}
    &
        \in H^0(\C, L_{1k}(d_{1k}))\setminus\{0\}, \text{ for }k=2,3,4.
\end{split}
\end{align}
The anticanonical degree of our new vector $\bm{b}$ is
\[
d_{12}+d_2+(d_{23}-e)+d_3+(d_{34}-e)=d-2e
\]
and we can use \cref{Lemma: GeneralUpperBound-DIM}
to bound the dimension of the parameter space of $\bm b$
by $d-2e + 4g$. 
This is already sufficient to deduce the bound on the dimension
since the motivic sum over $E_1$ will only increase the dimension by $e=\deg ( E_1 )$
and 
\[
d - \deg ( E_1 ) \leqslant d ( 1 - \delta_1 )
\]
in the case we are considering. 
For the point-counting measure when $k = \mathbf F_q$, 
using the divisor bound for the sum over $E_2,E_3, E_4, D_{i}$, we see that the quantity to be estimated is 
\[
\ll_\varepsilon 
q^{d\varepsilon}
\cdot 
\#_{\mathbf F_q}
\left ( 
\sum_{
    \bm{a}'\in \mathcal{H}_{\bm{d}'}^*
}
\sum_{\substack{E_1\leq \ddiv(a_1) \\ \deg(E_1)\geq \delta_1 d}}
\left [
\bm a'' \in \mathcal{H}^\bullet_{\bm{d}''} \mid
\begin{array}{l}
    \eqref{equation-pluckers-relations}\text{ and }
    E_1 \leqslant \ddiv(a_{jk}), \, 2\leqslant j, k\leqslant 4 
\end{array}
\right ]
\right )
\]
hence we can use the same reparameterisation that led to \eqref{NewBies}. 
Forgetting about the coprimality conditions on $\bm a'$, we thus arrive at the estimate
\begin{align*}
&
\ll_\varepsilon q^{d\varepsilon}\sum_{e\geq \delta_1 d}
\sum_{L\in \PPic^0(\C)}\sum_{\substack{E_1\in \DDiv(\C) \\ \deg(E_1)=e \\ \phi(E_1)=L}}
\#_{\mathbf F_q} \{ 
 \bm b \mid \eqref{equation-pluckers-relations} \text{ and } \eqref{NewBies} \} \\
& 
\ll_\varepsilon q^{d(1+\varepsilon)}
\sum_{e\geq \delta_1 d}
{ \sum_{\substack{E_1\in \DDiv(\C) \\ \deg(E_1)=e}}q^{-2e}} \\
& 
\ll_\varepsilon q^{d(1-\delta_1 +\varepsilon)},
\end{align*}
where the first estimate follows from \cref{Prop: GeneralUpperBound} and the second one is standard. This is satisfactory for \cref{Le: LargeMoebius.MainBody}.

For the case $\max \deg(D_{i})\geqslant \delta_1 d$ we use a similar idea. Again, we may assume that $\deg(D_1)\geq \delta_1 d$. 
Let $e=\deg(D_1)$, $L=\phi(D_1)$ and take a non-zero global section $u\in H^0(\C, L(e))$. This time we use the change of coordinates
    \[
    \bm b  = \bm u \cdot ( \bm a ' , \bm a '' ) 
    \]
    with $\bm{u}=(1, u, 1,1,1) $. 
We then have
    \begin{align}
\begin{split}\label{NewBies2}
b_1&\in H^0(\C, L_1\otimes L(d_1+e))\setminus\{0\},\\
b_{1k}&\in H^0(\C, L_{1k}\otimes L^{-1} (d_{1k}-e))\setminus\{0\}, \text{ for }k=2,3,4, \\
b_i&\in H^0(\C, L_i(d_i))\setminus\{0\} \text{ for }i=2,3,4,\\
b_{ij}&\in H^0(\C,L_{ij}(d_{ij}))\setminus\{0\} \text{ for }2\leqslant i,j\leqslant 4,\\
\end{split}
\end{align}
and  
\[
D_1\leqslant \ddiv(b_1).
\]
Note that crucially the spaces in \eqref{NewBies2} only depend on $e=\deg(D_1)$ and $L=\phi(D_1)$.  The anticanonical height of $\bm{b}$ is  
        \[
    (d_{12}-e)+d_2+d_{23}+d_3+d_{34}=d-e. 
    \]
In particular, since the dimension of the space of $D_1\in \DDiv(\C)$ such that $D_1\leq \ddiv(b_1)$ is 0, for a fixed $L\in \PPic^0(\C)$ by \cref{Prop: GeneralUpperBound} the dimension of the relevant parameter space of $\bm{b}$ is 
\[
\leq \max_{e\geq \delta_1 d }(d-e +4g) = d(1-\delta_1) + 4g,
\]
which is satisfactory for the dimensional bound upon taking into account that $\dim \PPic^0 ( \mathscr C )=g$. 

    For the counting, we first apply the divisor estimate to the sum over $D_2,D_3,D_4, E_1,\dots, E_4$. 
    The quantity to be estimated is then
    \begin{align*}
    & \ll_\varepsilon q^{d\varepsilon}
    \sum_{L\in \PPic^0(\C)}
    \sum_{e\geq \delta_1 d} \sum_{ \bm b \colon \eqref{NewBies2}}\sum_{\substack{D_1\in \DDiv(\C) \\ \deg(D_1)=e \\ D_1\leqslant \ddiv(b_1)}} \mathbf 1_{\eqref{equation-pluckers-relations}} ( \bm b )  \\
    & \ll_\varepsilon q^{{2d\varepsilon}}
    \sum_{L\in \PPic^0(\C)}
    \sum_{e\geq \delta_1 d} \#_{\mathbf F_q} \{ 
 \bm b \mid \eqref{equation-pluckers-relations} \text{ and } \eqref{NewBies2} \} \\
    & \ll_\varepsilon q^{{d(1+2\varepsilon)}}
    \sum_{e\geq \delta_1 d}q^{-e}\\
    & \ll_\varepsilon q^{d(1-\delta_1 +2\varepsilon)},
    \end{align*}
where the penultimate estimate follows from \cref{Prop: GeneralUpperBound}. This is satisfactory upon redefining $\varepsilon$ and thus completes our proof.

\end{proof}


\section{The main term}\label{Se: MainTerm}

For any $0<\delta_1 <1$, combining \eqref{Eq: Decomp.M(d,L).Mainbody} with \cref{Le: M(a')=moebius.Inversion} and \cref{Le: LargeMoebius.MainBody} gives
\begin{equation}\label{Eq: LargeMoebiusRemoved}
[ M
    ( \bm{d},\underline L )
]
= 
    \sum_{
    \bm{a}' \in \mathcal{H}_{\bm{d}'}^*
    ( \underline L' )
    }
    M ( \bm{a}',\bm{d},\underline L,\delta_1 ) 
+ E ( \delta_1 )
\end{equation}
with $\dim ( E ( {\delta_1} ) ) \leqslant d ( 1 - \delta_1 ) +O_g(1)$
and $\#_{\mathbf F_q} E ( {\delta_1} ) = O_{\varepsilon} (q^{d(1-\delta_1 + \varepsilon )} ) $
where 
\[
M ( \bm{d},\underline L ; \bm{a}' ,  \delta_1 ) 
= 
\sum_{ 
\bm D , \bm E \in \DDiv ( \mathscr C )^4_{\leqslant \delta_1 d} }
\mu_\mathscr C ( \bm D, \bm E )
\left [
\bm a'' \in 
\mathcal H^\bullet_{\bm{d}''} ( \underline L'' ) \; \left | \;
\begin{array}{l}
        ( \bm a ' , \bm a '' ) \colon 
        \eqref{equation-pluckers-relations}, \\ 
        \bm F \leqslant \ddiv( \bm a '') 
    \end{array} \right.
\right ]
\]
in $\KVar{\mathcal H^\bullet_{\bm{d}'}
    \times \PPic^0 ( \mathscr C )^{10}}$. 


\subsection{Decomposition}

By \cref{Le: CongruenceImpliesPluecker}, we know that if $\bm{a}'\in \mathcal{H}^*_{\bm{d}'}$ a tuple $\bm{a}''\in \mathcal{H}^\bullet_{\bm{d}''}$

satisfies the Plücker equations \eqref{equation-pluckers-relations} if and only if 
\begin{align*}
        \ddiv(a_1)&\leqslant \ddiv(a_3a_{34}-a_2a_{24}),\\
        \ddiv(a_2)&\leqslant \ddiv(a_4a_{34}+a_1a_{13}),
    \end{align*}
in which case $a_{12}, a_{14}, a_{23}$ are uniquely determined by the remaining variables. Furthermore, in the definition of $M(\bm{d},\underline L; \bm{a}', \delta_1)$ we have the additional requirement that 
\begin{equation}
    F_{ij}\leqslant \ddiv(a_{ij})
\end{equation}
for every $1\leqslant i,j \leqslant 4$. We can merge these two conditions to only have conditions on $a_{13}, a_{24}, a_{34}$: 
\begin{align}\label{Eq: CongConditions.MainBody}
    \begin{split}
        \ddiv(a_1)+F_{14}&\leqslant \ddiv(a_3a_{34}-a_2a_{24}), \\
        \ddiv(a_2)+F_{23}&\leqslant \ddiv(a_4a_{34}+a_1a_{13}),\\
        F_{ij}&\leqslant \ddiv(a_{ij}) \text{ for }{ij}\in\{13,24,34\}.
    \end{split}
\end{align}
Following the discussion in \cref{SubSe: VecBundAdele}, the conditions \eqref{Eq: CongConditions.MainBody} define an algebraic family of 
vector bundles
\[
\E 
= 
\E  ( \bm{a}', \bm{d}'', \bm{D}, \bm{E} )
\]
parameterised by $ \mathcal H^*_{\bm{d}'} ( \underline L') \times 
\DDiv ( \mathscr C )^4 \times \DDiv ( \mathscr C )^4 
\ni ( \bm a ' , \bm D , \bm E )
$
that are subbundles of 
\[
    L_{13}(d_{13})
    \oplus L_{24}(d_{24})
    \oplus L_{34}(d_{34})
\]
whose global sections are precisely 
\begin{align*}
H^0(\C, \E(\bm{a}', \bm{d}'', \bm{D}, \bm{E})) = 
\left\{(a_{ij})\in \bigoplus_{ij\in \{13,24,34\}}H^0(\C, L_{ij}(d_{ij}))\; \left|  \; \eqref{Eq: CongConditions.MainBody}\right.\right\}.     
\end{align*}
Therefore, we have
\begin{equation}\label{Eq: M(a'delta)=Main+Error}
 \sum_{
    \bm{a}'\in \mathcal{H}_{\bm{d}'}^*
    ( \underline L_I ) }
    M (\bm{d},\underline L ; \bm{a}' ,  \delta_1 ) 
    =  
    M_0(\bm{d},\underline L , \delta_1)
    -
    E_0(\bm{d},\underline L , \delta_1),
\end{equation}
where 
\begin{equation}
\label{Eq: MainTermI.MainBody}
    M_0 
    (\bm{d}, \underline L ; \delta_1)
    = 
    \sum_{\bm{a}' \in \mathcal{H}_{\bm{d}'}^* ( \underline L' )
        }
    \sum_{
        \bm{D},\bm{E}\in \DDiv(\C)_{\leqslant \delta_1 d}^4}    
    \mu_\C(\bm{D},\bm{E})
    \mathbf L^{h^0(\C, \E( \bm{a}',\bm{d}'', \bm{D},\bm{E}))}
\end{equation}
will deliver the main term and
\begin{equation}\label{Eq: ErrorTermI.MainBody}
    E_0(\bm{d}, \underline L ; \delta_1 ) 
    = 
    \sum_{\bm{a}' \in \mathcal{H}_{\bm{d}'}^* ( \underline L' )
        }
    \sum_{
        \bm{D},\bm{E}\in \DDiv(\C)_{\leqslant \delta_1 d}^4}
    \left[ \bm{a}''\in \mathcal{H}_{\bm{d}''}
    \; \left | \; 
    \begin{array}{l}
        \eqref{equation-pluckers-relations} \text{ and }       \bm F \leqslant \ddiv( \bm a ''), \\\prod_{ij}a_{ij}=0
    \end{array}
    \right.
    \right]
\end{equation}
is the contribution where one of the $a_{ij}$ vanishes. We will provide upper bounds for the latter in \cref{Se: ErrorTerms} and focus our attention on the main term first. From now on it will be convenient to drop $\underline L$ from our notations. 

We will next replace $h^0(\C, \E)$ in the definition $M_0(\bm{d},\delta_1)$ by its expected value 
\[
M_0(\bm{d},\delta_1)= M_1(\bm{d}, \delta_1) +E_1(\bm{d},\delta_1),
\]
where 
\[
M_1(\bm{d},\delta_1) 
= 
\sum_{\bm{a}'\in \mathcal{H}^*_{\bm{d}'} }
{
    \sum_{
        \substack{\bm{D}\in \DDiv(\C) \\ \deg(D_i)\leqslant \delta_1 d}
    }
    \sum_{
        \substack{\bm{E}\in \DDiv(\C) \\ \deg(E_i)\leqslant \delta_1 d}
    }
    }
    \mu_\C(\bm{D},\bm{E})
\mathbf L^{3(\mu(\E(\bm{a}', \bm{d}'', \bm{D},\bm{E}))+1-g)} 
\]
and 
\begin{equation}\label{Eq: Defi.E1}
E_1(\bm{d},\delta_1)= \sum_{\bm{a}'\in \mathcal{H}^*_{\bm{d}'} }
{
    \sum_{
        \substack{\bm{D}\in \DDiv(\C) \\ \deg(D_i)\leqslant \delta_1 d}
    }
    \sum_{
        \substack{\bm{E}\in \DDiv(\C) \\ \deg(E_i)\leqslant \delta_1 d}
    }
    }
    \mu_\C(\bm{D},\bm{E})
\left(\mathbf{L}^{h^0(\C, \E(\bm{a}', \bm{d}'',\bm{D},\bm{E}))}-\mathbf L^{3(\mu(\E(\bm{a}', \bm{d}'', \bm{D},\bm{E}))+1-g)}\right)
\end{equation}
The remainder of this is concerned with evaluating $M_1(\bm{d},\delta_1)$ asymptotically.

\subsection{A degree computation}
Given $\bm D , \bm E \in\DDiv(\C)^4$, 
let us define
\begin{align*}
B=E_1+E_2+E_3+E_4+[D_1;D_2;D_3]+[D_1;D_2;D_4]+[(D_1;D_2);D_3;D_4]
\end{align*}

and write $b=\deg(B)$.
\begin{lemma}\label{lemma-degree-main-vector-bundle}
For every $\bm a ' , \bm d '' , \bm D , \bm E$ as before,
we have 
    \[
     \deg \mathcal E ( \bm a ' , \bm d '' , \bm D , \bm E ) 
    =
    d-|\bm d'| 
    - b  
    \]
where $|\bm d'| = d_1 + d_2 + d_3 + d_4$. 

In particular,

\[
\deg \mathcal E ( \bm a ' , \bm d '' , \bm D , \bm E ) 
\geqslant 
d(1/5-13\delta_1) . 
\]
\end{lemma}

\begin{proof} 

    To keep the notation simple, we shall write $\E$ instead of $\E(\bm{a}',\bm{d}'',\bm{D},\bm{E})$ in the course of the proof.
    
    Let $N$ be a sufficiently large positive integer. In particular, we assume that $h^1(\C,\E(N))=0$. 
    By Riemann--Roch we have 
    \begin{equation}\label{Eq: deg.via.H0}
    \deg(\E)=\deg(\E(N))-3N = h^0(\C,\E(N))-3+3g-3N
    \end{equation}
    and we may thus focus on computing $h^0(\C, \E(N))$ assuming $N$ is as large as we wish. 

    First of all, since $E_2\leq \ddiv(a_2)$ and $F_{34}\leq \ddiv(a_{34})$ implies $E_2\leq \ddiv(a_{34})$, the condition 
    \[
    \ddiv(a_1)+F_{14}\leq \ddiv(a_3a_{34}-a_2a_{24})
    \]
    already holds if 
    \[
    \ddiv(a_1)+F_{14}-E_2 \leq \ddiv(a_{3}a_{34}-a_2a_{24})\quad\text{ and }\quad F_{34}\leq \ddiv(a_{34}).
    \]
    We can repeat the argument with $E_3, E_1,E_4$ to see that \eqref{Eq: CongConditions.MainBody} is equivalent to the a priori weaker conditions 
    \begin{align}
        \begin{split}\label{Eq: NewCong}
            \ddiv(a_1)+[D_1;D_4]&\leq \ddiv(a_3a_{34}-a_2a_{24}),\\
            \ddiv(a_2)+[D_2;D_3] &\leq \ddiv(a_4a_{34}+a_1a_{13}),\\
            F_{ij}& \leq \ddiv(a_{ij})\text{ for }ij\in\{13,24,34\}.
        \end{split}
    \end{align}
    We will fix $a_{34}$ and then determine the dimension of pairs $(a_{13},a_{24})$ for which \eqref{Eq: NewCong} holds.

    However, we begin with determining when for a fixed $a_{34}$ the system  \eqref{Eq: NewCong} has a solution. There exists an $a_{24}$ such that \eqref{Eq: NewCong} holds if and only if
    \[
    (F_{24}; \ddiv(a_1)+[D_1;D_4])\leq \ddiv(a_3a_{34}).
    \]
    Since $F_{24}=[D_2;D_4]+E_1+E_3$ the conditions $E_1\leq \ddiv(a_1)$ and $(E_i;D_j)=0$ show that 
    \begin{equation}\label{Eq: GCD1}
    (F_{24}; \ddiv(a_1)+[D_1;D_4])=E_1+([D_2;D_4];[D_1;D_4])=E_1+(D_4;[D_1;D_2]).    
    \end{equation}
    We then also have 
    \begin{align*}
    [F_{24}; \ddiv(a_1)+[D_1;D_4]] & =[D_2;D_4]+E_1+E_3+\ddiv(a_1)+[D_1;D_4]-(E_1+(D_4;[D_1;D_2]))\\
    &= \ddiv(a_1)+E_3+[D_1;D_2;D_4].
    \end{align*}
    The solutions $a_{24}\in H^0(\C, \E(N))$ to \eqref{Eq: NewCong} are parameterised by a translate of an affine space of dimension 
    \begin{align}
    \begin{split}\label{Eq: Dim.a24}
    h^0(\C, L_{24}(d_{24}+&N) \otimes \OK_\C(-[F_{24}; \ddiv(a_1)+[D_1;D_4]])) \\
    &= d_{24}+N-d_1+\deg(E_3)+\deg([D_1;D_2;D_4])+1-g,
    \end{split}
    \end{align}
    assuming $N$ is sufficiently large. 
    
    We may repeat the exact same argument for $a_{13}$ to see that \eqref{Eq: NewCong} has a solution in $a_{13}$ if and only if 
    \[
    E_2+(D_3;[D_1;D_2])\leq \ddiv(a_{34}),
    \]
    in which case the solution space is isomorphic to an affine space of dimension 
    \begin{equation}\label{Eq: dim.a13}
    d_{13}+N-d_2-\deg(E_4)-\deg([D_1;D_2;D_3])+1-g.
    \end{equation}
    We are thus left with computing the dimension of $a_{34}\in H^0(\C,L_{34}(d_{34}+N))$ for which \eqref{Eq: NewCong} has a solution. We have just seen this is exactly the case if and only if 
    \[
    E_1+(D_4;[D_1;D_2])\leq \ddiv(a_{34}) \quad \text{and}\quad E_2+(D_3;[D_1;D_2])\leq \ddiv(a_{34}).
    \]
    Using $F_{34}=[D_3;D_4]+E_1+E_2$ together with $(D_i;E_j)=0$, we obtain 
    \[
    [F_{34};E_1+(D_4;[D_1;D_2]);E_2+(D_3;[D_1;D_2])]=E_1+E_2+[D_3;D_4;(D_1;D_2)]
    \]
and hence the dimension of the $a_{34}$ is 
\begin{align}
\begin{split}\label{Eq:33333}
h^0(\C, L_{34}(d_{34}+&N)\otimes \OK_\C(-(E_1+E_2+[D_3;D_4;(D_1;D_2)])))\\
&= d_{34}+N-\deg(E_1)-\deg(E_2)-\deg([D_3;D_4;(D_1;D_2)])+1-g,
\end{split}
\end{align}
again assuming $N$ is sufficiently large.

Combining \eqref{Eq: Dim.a24}, \eqref{Eq: dim.a13} and \eqref{Eq:33333} gives 
\begin{align*}
h^0(\C,\E(N)) &=  d_{24}+N-d_1+\deg(E_3)+\deg([D_1;D_2;D_4])+1-g \\
               &\phantom{=}+  d_{13}+N-d_2-\deg(E_4)-\deg([D_1;D_2;D_3])+1-g\\
               &\phantom{=}+ d_{34}+N-\deg(E_1)-\deg(E_2)-\deg([D_3;D_4;(D_1;D_2)])+1-g
             \\
             & =  d_{13}+d_{24}+d_{34}-d_1-d_2-\deg([D_1;D_2;D_3])-\deg([D_1;D_2;D_4])\\
             & \phantom{=}-\deg([D_3;D_4;(D_1;D_2)])+3N-3(1-g).     
\end{align*}
This concludes the computation of $\deg(\E)$ by \eqref{Eq: deg.via.H0} upon observing that 
\[
d_{13}+d_{24}+d_{34}-d_1-d_2=d-d_1-d_2-d_3-d_4.
\]
It thus remains to establish the claimed lower bound on $\deg(\E)$. For this, recall that $\deg(D_i), \deg(E_j)\leq \delta_1 d$ for all $i,j\in \{1,2,3,4\}$. This easily gives $b \leq 13\delta_1 d$. In addition, by \cref{Cor: Upperbounds.di} we have $d_1+d_2+d_3+d_4\leq 4d/5$ and hence 
\[
\deg(\E) = d-d_1-d_2-d_3-d_4-b \geq d/5-13\delta_1 d
\]
as wanted. 

\end{proof}


\subsection{An arithmetic function}

Having computed $\deg(\E(\bm{a}',\bm{d}'',\bm{D},\bm{E}))$, we can now continue our analysis of $M_1(\bm{d},\delta_1)$. Inserting \cref{lemma-degree-main-vector-bundle} into the definition of $M_1(\bm{d},\delta_1)$ gives

\[
M_1(\bm{d},\delta_1)=\sum_{
\bm a' \in\mathcal H^*_{\bm d'} ( \underline L ' )
}
\sum_{
        \bm{D},\bm{E}\in \DDiv(\C)_{\leqslant \delta_1 d}^4}
\mu_\mathscr C 
( \bm D , \bm E ) 
\mathbf L_k^{d-|\bm d'|-b + 3 ( 1 - g ) .
},
\]
Let us consider the motivic arithmetic function
\begin{align*}
\mathfrak D ( \bm a ' )
& =
\sum_{
\substack{
\bm D \in \DDiv ( \mathscr C )^4 \\
\forall i \neq j \; ( \ddiv ( a_i ) ; D_j ) = 1 
}
}
\mu_\mathscr C ( \bm D )
\mathbf L_k^{
- \deg ( [ (D_1 ; D_2);D_3;D_4 ] ) 
- \deg ( [ D_1 ; D_2 ; D_3 ] ) 
- \deg ( [ D_1 ; D_2 ; D_4 ] ) }\\
& =
\prod_{v \mid a_1 a_2 a_3 a_4}
\left ( 
1 - \mathbf L_v^{-2}
\right )
\prod_{v \nmid a_1 a_2 a_3 a_4}
\left ( 
1 - 4 \mathbf L_v^{-2} + 3 \mathbf L_v^{-3}
\right )
\end{align*}

which depends only on $\bm A ' = \ddiv ( \bm a ' ) $.
It is well defined, since the local factor of the second product is a polynomial in $\mathbf L^{-1}$ of valuation $2$. 
Set 
\[
\theta ( \bm A '  )
=
\mathfrak D ( \bm A ' ) 
\prod_{i=1}^4 
\mathfrak E ( A_i )
\]
where 
\[
\mathfrak E ( A )
=
\prod_{v \mid a}
\left ( 
1 - \mathbf L_v^{-1} 
\right ) 
=
\sum_{E \leqslant A}
\mu_\mathscr C ( E ) \mathbf L_k^{- \deg ( E )} 
\]
for any $A = \ddiv ( a ) $
so that
\begin{align*}
& \sum_{
\substack{
\bm D , \bm E  \in \DDiv ( \mathscr C )^4 \\
\bm E \leqslant \bm A \\
\forall i \neq j \; ( A_i ; D_j ) = 1 
}
}
\mu_\mathscr C 
( \bm D , \bm E ) 
\mathbf L_k^{
- ( \deg ( E_1 ) 
+ \deg ( E_2 )
+ \deg ( E_3 )
+ \deg ( E_4 )
+ \deg ( [ (D_1 ; D_2) ;  D_3 ; D_4  ] ) 
+ \deg ( [ D_1 ; D_2 ; D_3 ] ) 
+ \deg ( [ D_1 ; D_2 ; D_4 ] )
)
}\\
& 
= 
\sum_{
\substack{
\bm D \in \DDiv ( \mathscr C )^4 \\
\forall i \neq j \; ( A_i ; D_j ) = 1 
}
}
\mu_\mathscr C ( \bm D )
\mathbf L_k^{
- \deg ( [ (D_1 ; D_2); D_3; D_4 ] ) 
- \deg ( [ D_1 ; D_2 ; D_3 ] ) 
- \deg ( [ D_1 ; D_2 ; D_4 ] ) }
\\
& \qquad \times 
\sum_{E_1 \leqslant A_1}
\mu_\mathscr C ( E_1 ) \mathbf L_k^{- \deg ( E_1 )}
\times ... 
\times 
\sum_{E_4 \leqslant A_4}
\mu_\mathscr C ( E_1 ) \mathbf L_k^{- \deg ( E_4 )}
\\
& 
= 
\mathfrak D ( \bm A ' ) 
\mathfrak E ( \bm A' )
\\
& 
= \theta ( \bm A ' )  
\end{align*}
where it is convenient to 
write $\mathfrak E ( \bm A' ) = \prod_{i=1}^4 
\mathfrak E ( A_i ) $. 

We can similarly define the motivic functions 
\[
\left ( 
A \mapsto 
\mathfrak E ( A , \delta_1 )
= 
\sum_{
    \substack{
    E \leqslant A
    \\
    \deg ( E ) \leqslant \delta_1 d
    }
    }
\mu_\mathscr C ( E ) \mathbf L_k^{- \deg ( E )}
\right )
\in 
\mathscr M_{\DDiv ( \mathscr C )}, 
\]

\[
\left ( 
    \bm A' \mapsto 
    \mathfrak E ( \bm A ' , \delta_1 )
    = 
    \prod_{i=1}^4
       \mathfrak E ( A_i , \delta_1 )
\right )
\in 
\mathscr M_{\DDiv ( \mathscr C )^4} , 
\]
\[
\left ( 
 \bm A ' \mapsto 
\mathfrak D ( \bm A ' )
=
\sum_{
    \substack{
    \bm D \in \DDiv ( \mathscr C )^4_{\leqslant \delta_1 d} \\
    \forall i \neq j \; (  A_i  ; D_j ) = 1 
}
}
\mu_\mathscr C ( \bm D )
\mathbf L_k^{
- \deg ( [ (D_1 ; D_2); D_3 ; D_4 ] ) 
- \deg ( [ D_1 ; D_2 ; D_3 ] ) 
- \deg ( [ D_1 ; D_2 ; D_4 ] ) }
\right )
\in 
\mathscr M_{\DDiv ( \mathscr C )^4} 
,
\]

and 
\[
\left ( 
    \bm A ' \mapsto 
    \theta ( \bm A ' , \delta_1 ) 
    =
    \mathfrak D ( \bm A ' , \delta_1 )
    \mathfrak E ( \bm A' , \delta_1 )
\right )
\in \mathscr M_{\DDiv ( \mathscr C )^4} .
\]
A similar computation shows that this last quantity
equals 
\[
\sum_{
\substack{
 \bm D , \bm E  \in \DDiv ( \mathscr C )^4_{\leqslant \delta_1 d} \\
\bm E \leqslant \bm A \\
\forall i \neq j \; ( A_i ; D_j ) = 1 
}
}
\mu_\mathscr C 
( \bm D , \bm E ) 
\mathbf L_k^{
            -\sum_{i=1}^4\deg(E_i)- \deg([(D_1;D_2);D_3;D_4])-\deg([D_1;D_2;D_3])-\deg([D_1;D_2;D_4])
            }. 
\]

\begin{lemma}\label{Eq: MainTerm.RemoveTruncation}
For every $\bm a'$,
\[
\dim \left ( 
\theta (  \bm A ' , \delta_1 ) 
-
\theta ( \bm A ' )
\right )
\leqslant - \delta_1 d  
\]
in $\mathscr M_{\DDiv ( \mathscr C )^4}$. 
Moreover,
if $k=\mathbf F_q$,
then for every $\bm A ' \in \DDiv ( \mathscr C )^4 ( \mathbf F_q)$,
\[
\#_{\mathbf F_q}
    \theta ( \bm A ' , \delta_1 )
=
\#_{\mathbf F_q}
    \theta ( \bm A ' )
+ 
O ( q^{-\delta_1 d /2} ) 
\]
\end{lemma}

\begin{proof}
    We can see from their definition as Euler products that $\mathfrak{D}$ and $\mathfrak{E}$ are uniformly bounded, both for the point-counting measure (by one) and the dimensional filtration (they are zero-dimensional objects). 
    Therefore it suffices to estimate the differences coming from the truncations in both functions individually. On one hand we have
    \[
    \mathfrak E (A) -
    \mathfrak{E} (A, \delta_1 )
    =
    \sum_{\substack{E \le A \\ \deg (E) > \delta_1 d}}\mathbf L_k^{-\deg(E)} 
    \]
    and 
    \[
    \#_{\mathbf F_q}  \sum_{\substack{E \le A \\ \deg (E) > \delta_1 d}}\mathbf L_k^{-\deg(E)} \ll \tau(a) \cdot q^{-\delta_1 d} \ll q^{-\delta_1d/2}
    \]
    while the dimension is bounded by $-\delta_1 d$. 
    On the other hand, by symmetry we have
    \[
    \dim ( 
    \mathfrak{D}(\bm A', \delta_1)-\mathfrak{D}(\bm A') ) 
    \leqslant - \delta_1 d 
    \]
    and 
    \[
    | \#_{\mathbf F_q} ( 
    \mathfrak{D}(\bm A', \delta_1)-\mathfrak{D}(\bm A') ) | 
    \ll \sum_{\substack{\bm D \in \Div(\mathscr C )^4 \\\deg(D_1)>\delta_1 d}} q^{
- \deg ( [ (D_1 ; D_2); D_3;D_4 ] ) 
- \deg ( [ D_1 ; D_2 ; D_3 ] ) 
- \deg ( [ D_1 ; D_2 ; D_4 ] ) }.
    \]
    By an application of Rankin's trick, this expression is bounded by
    \[
    q^{-c\delta_1d} \sum_{\bm D \in \Div(\mathscr C )^4}  q^{c\deg(D_1)
- \deg ( [ (D_1 ; D_2) ; D_3 ; D_4 ] ) 
- \deg ( [ D_1 ; D_2 ; D_3 ] ) 
- \deg ( [ D_1 ; D_2 ; D_4 ] ) }.
    \]
for any $c>0$. But as in the discussion of $\mathfrak{D}(\bm A')$, the expression on the right-hand side can now be decomposed into an Euler product $\prod_v \left(1+O( q_v^{c-2})\right)$ which is uniformly bounded for any fixed $c<1$. Choosing $c=1/2$ we obtain the result.

\end{proof}

Recall that we assumed that our numbering is such that $d_1 \leqslant d_2 \leqslant d_3 \leqslant d_4$. 
Let $\theta_0$ be the motivic function on $\DDiv ( \mathscr C )^4$ given by the gcd conditions for the first 4 variables.

Thanks to \cref{lemma-main-term-general-error-estimate-Fq} 
and \cref{lemma-main-term-general-error-estimate-DIM} we have the following. 

\begin{proposition}\label{Prop: EvalMainTerm}
   Let $\varepsilon \in ( 0 , 1/2)$. We have
\begin{align*}
&\mathbf L_k^{
- d_1 - d_2 - d_3 - d_4
} 
\sum_{
\bm A ' \in \DDiv^{\bm d '} ( \mathscr C ) 
}
\theta_0 ( \bm A '  )
\theta ( \bm A '  )
 \\
& = 
\left ( 
\frac{\left [ \PPic^0 ( \mathscr C ) \right ] \mathbf L^{-g}}{1 - \mathbf L^{-1}}
\right )^4 
\prod_{v \in \mathscr C}
\left ( 
1 - \mathbf L_v^{-1}
\right )^5
\left ( 
    1 + 5 \mathbf L_v^{-1} + \mathbf L_v^{-2}
\right ) 
+ E ( \varepsilon , d_1 ) 
\end{align*}     
with $\dim ( E ( \varepsilon , d_1 )  ) \leqslant -\varepsilon d_1$.
If $k=\mathbf F_q$,
\begin{align*}
&q^{
- d_1 - d_2 - d_3 - d_4
} 
\#_{\mathbf F_q} 
\sum_{
\bm A ' \in \DDiv^{\bm d '} ( \mathscr C ) 
}
\theta_0 ( \bm A '  )
\theta ( \bm A '  )
 \\
& = 
\left ( 
\frac{\#_{\mathbf F_q} \PPic^0 ( \mathscr C ) q^{-g}}{1 - q^{-1}}
\right )^4 
\prod_{v \in \mathscr C}
\left ( 
1 - q_v^{-1}
\right )^5
\left ( 
    1 + 5 q_v^{-1} + q_v^{-2}
\right ) \\
& \qquad + 
O_\varepsilon ( q^{-\varepsilon d_1 } ) .
\end{align*}    
\end{proposition}

\begin{proof}
    Let us introduce
\begin{equation*}
    Z_{\theta} (T_1,T_2,T_3,T_4) 
    = \sum_{ \bm A ' \in \DDiv^{\bm d '} ( \mathscr C ) } 
        \theta_0 (\bm A') \theta (\bm A') 
        T_1^{\deg ( A_1 )} T_2^{\deg ( A_2 )} T_3^{\deg ( A_3 )}T_4^{\deg ( A_4 )}
\end{equation*}
so that the sum we want to estimate is exactly the degree $\bm d'$ coefficient 
of 
$ Z_{\theta}(T_1,T_2,T_3,T_4)$
normalised by $\mathbf L^{|\bm d ' |}$.
Define $\theta_\mu ( \bm A ' )$ to be the motivic arithmetic function on $\DDiv ( \mathscr C )^4$ given by the coefficients of 
\begin{equation*}
    Z_\mu ( T_1 , T_2 , T_3 , T_4 ) = 
    Z_{\theta}(T_1,T_2,T_3,T_4)\prod_{i=1}^4 
    Z_{ \mathscr C }^{\Kapr} (T_i)^{-1}
\end{equation*}
so that
\begin{equation*}
    \theta_0(\bm A')\theta(\bm A')=\sum_{\bm B' \leqslant \bm A'} \theta_{\mu}(\bm B')
\end{equation*}
(in other words, by Möbius inversion, justifying the notation with a $\mu$). 
By multiplicativity, $Z_\mu ( T_1 , T_2 , T_3 , T_4 )$ can be written as a motivic Euler product whose local factor is
\begin{equation*}
\prod_{i=1}^4 
\left(
1-T_i
\right) 
\cdot 
\left(
    \left(
        1- 4 \mathbf L^{-2} +3 \mathbf L^{-3}
    \right)
    +
    \left(
        1 - \mathbf L^{-1}
    \right)
    \left(
        1- \mathbf L^{-2}
    \right) 
    \left(
        \sum_{i=1}^4 \sum_{k\geqslant 1}
            T_i^k 
    \right)
\right) .
\end{equation*}
In particular, the motivic Euler product converges at $T_i = \mathbf L^{-1}$. To see this, one can replace $T_1, ... , T_4$ as well as every $\mathbf L^{-1}$ appearing in the local factor by a new indeterminate $T'$. One obtains a polynomial in $T'$ that has valuation at least $2$, hence the corresponding motivic Euler product converges at $T' = \mathbf L^{-1}$. In particular,
\begin{equation*}
    Z_\mu ( T_1 , T_2 , T_3 , T_4 )
    \vert _{T_1=T_2=T_3=T_4=\mathbf L^{-1}}
    \in \widehat{\mathscr M_k}^{\dim} 
\end{equation*}
is well-defined, as well as its point-counting counterpart. 
Specialising to 
$T_i= \mathbf L^{-1}$ 
(or $T_i = q^{-1}$ for the point-counting)
we precisely obtain
\[
\prod_{v\in \mathscr C}
\left (
    1 - \mathbf L_v^{-1} 
\right)^5
\left(
    1 + 5 \mathbf L_v^{-1} + \mathbf L_v^{-2}
\right) .
\]
We are now in position to apply \cref{lemma-main-term-general-error-estimate-Fq}
to (the point-couting measure of)
\begin{align*}
G ( \mathbf T ) & = Z_\theta ( T_1 , T_2 , T_3 , T_4 ) \\
& = Z_\mu ( T_1 , T_2 , T_3 , T_4 ) \prod_{i=1}^4 Z^\Kapr_\mathscr C ( T_i )
\end{align*}
and 
\begin{align*}
F ( \mathbf T ) 
& = Z_\theta ( T_1 , T_2 , T_3 , T_4 ) \prod_{i=1}^4 ( 1 - \mathbf L T_i ) \\
& = 
Z_\mu ( T_1 , T_2 , T_3 , T_4 ) 
\prod_{i=1}^4 ( 1 - \mathbf L T_i ) Z^\Kapr_\mathscr C ( T_ i ) . 
\end{align*}
Recalling that 
\[
\left [
( 1 - \mathbf L T ) Z^\Kapr_\mathscr C ( T )
\right ]_{T = \mathbf L^{-1}}
=
\frac{\left [ \PPic^0 ( \mathscr C ) \right ] \mathbf L^{-g}}{1 - \mathbf L^{-1}} 
\]
and 
that $d_1 = \min ( d_i )$, we obtain our claim. 
\end{proof}


\section{The error terms}\label{Se: ErrorTerms}
In this section, we show 
that the contributions coming from 
            $E_0(\bm{d},\delta_1)$ and $E_1(\bm{d},\delta_1)$ 
are both negligible with a power saving.


\subsection{The error term $E_0(\bm{d},\delta_1)$} 
The goal of this subsection is to prove the following estimate, where we recall
from \eqref{Eq: ErrorTermI.MainBody} given
that 
\[
  E_0(\bm{d} ; \delta_1 ) 
    = 
    \sum_{\bm{a}' \in \mathcal{H}_{\bm{d}'}^* ( \underline L' )
        }
    \sum_{
        \bm{D},\bm{E}\in \DDiv(\C)_{\leqslant \delta_1 d}^4}
    \left[ \bm{a}''\in \mathcal{H}_{\bm{d}''} 
        ( \underline L '' ) 
    \; \left | \; 
    \begin{array}{l}
        \eqref{equation-pluckers-relations} \text{ and }
        \bm F \leqslant \ddiv( \bm a ''), \\\prod_{ij}a_{ij}=0
    \end{array}
    \right.
    \right] 
\]
relatively to $\underline L \in \PPic^0 ( \mathscr C ) ^{10}$. 
\begin{proposition}\label{Prop: Bound.E0}
    There exists a constant $c_0 ( g )$  depending linearly on $g$ such that 
    \[
    \dim ( E_0 ( \bm d , \delta_1 ) ) 
    \leqslant 
    d - \min ( d_{ij} ) + 8\delta_1 d + c_0 ( g )
    \]
    and if $k=\mathbf F_q$,
    \[
    \#_{\mathbf F_q} 
    E_0 ( \bm d , \delta_1 )
    \ll_\varepsilon 
     q^{d ( 1 + 8\delta_1+\varepsilon ) - \min ( d_{ij} )} . 
    \]
\end{proposition}

\begin{proof}
    For the dimensional estimate, we are allowed to ignore the condition ${\bm F \leqslant \ddiv( \bm a '')}$ and perform the sum over $\bm{D}, \bm{E}$ trivially.
    We can assume without loss of generality 
    that $a_{12} = 0$. 
    Then the first equation in \eqref{equation-pluckers-relations} implies $a_4 a_{14} = a_3 a_{13}$
    and 
    \begin{align*}
       \frac{a_{14}}{a_3} 
            & = \frac{a_{13}}{a_4} \\
       \frac{a_{24}}{a_3} 
            & = \frac{a_{23}}{a_4}
    \end{align*}
    from which by the coprimality of the $a_i$ it follows that
        $a_{13} = a_4 x$, $a_{14} = a_3 x $,
        $a_{23} = a_4 y $ and $ a_{24} = a_3 y $
    for some $x\in H^0(\C, L_{13}(d_{13})\otimes \OK_\C(-\ddiv(a_4))),y\in H^0(\C, L_{24}(d_{24})\otimes \OK_\C(-\ddiv(a_3)))$. By the third equation of \eqref{equation-pluckers-relations} we also have 
    \[
    a_4 a_{34}  = a_2 a_{23} - a_1 a_{13} = a_4 ( a_2 y - a_1 x ) 
    \]
    hence $a_{34} = a_2 y - a_1 x$ is uniquely determined by $a_1, a_2, x, y$.
    In particular, it suffices to count $a_i$, $i\in \{ 1 , 2, 3 , 4 \}$ and $x,y$. 
    Up to a constant linearly depending on $g$ (which allows us to get rid of both the case of small degrees and of the $(1-g)$ coming from Riemann-Roch), the corresponding vector space has dimension 
    \begin{align*}
        d_1 + d_2 + d_3 + d_4 
        +
        d_{13} - d_4 + d_{24} - d_3 
        & = d_{31} + d_1 + d_{12} + d_2 + d_{24} 
        - d_{12} \\
        & = d - d_{12} \\
        & \leqslant
        d - \min ( d_{ij} )
    \end{align*}
    hence the result follows. 

\end{proof}

\subsection{The error term $E_1(\bm{d},\delta_1)$} 
Recall that \cref{Cor: Upperbounds.di}
tells us that we can always number our variables in such a way that
$d_1, d_2 \leq d/5$ and 
$d_1+d_2+d_3+d_4 \leq 4d/5$. 
We begin our treatment of $E_1(\bm{d},\delta_1)$ defined in \eqref{Eq: Defi.E1}, which turns out to be the most challenging part of this section. The goal is to establish the following estimate. 
\begin{proposition}\label{Prop: Error.E_1}
    Suppose that $0<\delta_1<1/65$. Then, there exists a constant $c_1 ( g )$ only depending linearly on $g$ such that 
    \[
    \dim ( E_1(\bm{d},\delta_1) ) 
    \leqslant 
    \max \left  ( 
        \frac{
        d(9+150\delta_1)
        }
        {10}
        ,
        d(1+2\delta_1 )-d_{34} 
    \right ) + c_1 ( g )
    \]
    and if $k = \mathbf F_q$,
    \[
    \#_{\mathbf F_q} E_1(\bm{d},\delta_1) \ll_{\varepsilon , g} q^{d(9+150\delta_1+\varepsilon)/10}+q^{d(1+2\delta_1 +\varepsilon)-d_{34}}.
    \]
\end{proposition}
By Riemann--Roch, we have 
\[
h^0(\C, \E(\bm{a}',\bm{d}'',\bm{D},\bm{E}) \ge 3(\mu(\E(\bm{a}',\bm{d}'',\bm{D},\bm{E}))+1-g)
\]
with equality unless $h^1(\C, \E(\bm{a}',\bm{d}'',\bm{D},\bm{E})> 0$.

In particular, we have that 
$\dim ( E_1(\bm{d},\delta_1) )$
and $\#_{F_q} ( E_1(\bm{d},\delta_1) )$
are controlled respectively by the $\dim$ and $\#_{F_q}$ of 

\begin{equation}
\label{eq:class-controlling-E1}
\mathop{\sum_{\bm{a}'\in \mathcal{H}^*_{\bm{d}'}}\sum_{\bm{D},\bm{E}\in \DDiv(\C)_{\leq \delta_1d}}}_{h^1(\C, \E(\bm{a}',\bm{d}'',\bm{D},\bm{E})) > 0}\bm{L}^{h^0(\C, \E(\bm{a}',\bm{d}'',\bm{D},\bm{E}))}.
\end{equation}

Let 
\[
0=\E_0\subset  \cdots \subset \E_r=
\E = 
\E (\bm{a}',\bm{d}'', \bm{D},\bm{E})
\]
be the Harder--Narasimhan filtration of $\E(\bm{a}',\bm{d}'',\bm{D},\bm{E})$ and denote its slopes by 
\[
\lambda_i = 
\lambda_i(\bm{a}',\bm{d}'',\bm{D},\bm{E})
\]
for $i
\in \{ 1,\dots ,r\}$. 
Note that since $\rk ( \E )=3$, we must have $r\leqslant 3$. 
In addition, \cref{lemma-degree-main-vector-bundle} tells us that 
\[
 \deg(\E(\bm{a}',\bm{d}'',\bm{D},\bm{E}) 
   \geqslant d(1/5-13\delta_1).
\]

In particular, since $r\lambda_1\geq \deg(\E(\bm{a}',\bm{d}'',\bm{D},\bm{E}))$,
both $\lambda_1$ and $\deg( \E )$
are non-negative 
as long as 
\[
\delta_1 < \frac{1}{65}, 
\]
which we shall henceforth assume.

Recall from \cref{Le: Upper.Bound.h0} that 
\begin{equation}\label{Eq: h0.bound.applied}
h^0 \left ( \C , \E ( \bm{a}',\bm{d}'',\bm{D},\bm{E} ) \right )
\leqslant 
    \sum_{i=1}^r
        \max \left \{
            \rk (\E_i/\E_{i-1}) (\lambda_i+1-g) ,
            \rk (\E_i/\E_{i-1})g
        \right \}.
\end{equation}
Let us first deal with the contribution from those $\bm{a}'$ and $\bm{D},\bm{E}$ with $\lambda_1\leqslant 2g-2$. 
In this case, \eqref{Eq: h0.bound.applied} above 
gives $h^0(\C, \E(\bm{a}',\bm{d}'',\bm{D},\bm{E}))\leqslant 3g$. 

Hence 
plugging that into \eqref{eq:class-controlling-E1} we obtain that
the total contribution 
of $\#_{\mathbf F_q} E_1 ( \bm d , \delta_1 )$ 
is at most
\begin{align*}q^{3g}\sum_{\substack{\bm{a}'\in \mathcal{H}^\bullet_{d_I}}}\sum_{\bm{D}\in \DDiv(\C)_{\leqslant \delta_1 d}^4 }
    \sum_{\bm{E}\in \DDiv(\C)_{\leqslant \delta_1 d}^4} 1    
    & \ll q^{3g + d_1+d_2+d_3+d_4+8\delta_1 d}\\
    & \ll q^{3g + d(4/5 + 8\delta_1)},
\end{align*}

and similarly for the dimension,
which is satisfactory under our assumption that $\delta_1< 1/65$.

Therefore, from now on we assume that $\lambda_1> 2g-2$. 
Note that as we have $\lambda_r\leqslant 2g-2$ by \cref{Le: h^1not0implieslambda_rsmall}, this implies in particular that 
\[
r \in \{ 2 , 3 \}. 
\]

Observe that since $\mu(\E_1(2g-\lambda_1))=2g-\lambda_1+\lambda_1>g-1$ and $\E_1$ is a subbundle of $\E = \E(\bm{a}',\bm{d}'',\bm{D},\bm{E})$, we have 
\begin{equation}\label{Eq: h1>0}
h^0
( 
    \C, \mathcal{E}(2g-\lambda_1))
)
>0. 
\end{equation}
Next, we split 
\[
\sum_{\bm{a}'\in \mathcal{H}^*_{\bm{d}'} }
\mathop{
    \sum_{
        \bm{D} \in \DDiv(\C)_{\leqslant \delta_1 d}^4 
    }
    \sum_{
        \bm{E}\in \DDiv(\C)_{\leqslant \delta_1 d}^4 
    }
    }_{
        \substack{
        h^1(\C, \mathcal{E}(\bm{a}', \bm{d}'', \bm{D},\bm{E})) > 0  \\
        \lambda_1 ( \mathcal{E}(\bm{a}', \bm{d}'', \bm{D},\bm{E}) ) > 2g-2 }
    }
    \mu_\C ( \bm{D}, \bm{E} )
\mathbf L^{h^0(\C, \mathcal{E}(\bm{a}',\bm{d}'',\bm{D},\bm{E}))}
\]
into two classes corresponding to the following two subcases: let
\begin{enumerate}
    \item $E_{1,1}(\bm{d},\delta_1)$ be the contribution, from those $\bm{a}',\bm{D},\bm{E}$ for which there exists 
    \[
    (a_{ij})\in H^0(\C, \E_1(2g-\lambda_1))
    \]
    such that $a_{34}\neq 0$;
    \item and $E_{1,2}(\bm{d},\delta_1)$ be the remaining contribution. 
\end{enumerate}
\cref{Prop: Error.E_1} will follow from the following two estimates.
\begin{lemma}\label{Le: E_11}
    There exists a constant $c_{1,1} ( g )$ depending linearly on $g$ such that  
    \[
    \dim E_{1,1} 
        ( \bm d , \delta_1 ) 
    \leqslant 
        d \frac{( 9 + 85 \delta_1 )}{10} + c_{1,1} ( g )
    \]
    and if $k = \mathbf F_q$,
    \[
    \#_{\mathbf F_q} E_{1,1}(\bm{d},\delta_1) \ll_{\varepsilon , g}q^{d(9+85\delta_1 +\varepsilon)/10}.
    \]
\end{lemma}
\begin{lemma}\label{Le: E_12}
    There exists a constant $c_{1,2} ( g )$ depending linearly on $g$ such that  
    \[
    \dim E_{1,2} ( \bm d , \delta_1 )
        \leqslant 
        d(1+15\delta_1)-d_{34} + 
        c_{1,2} ( g )
    \]
    and if $k =\mathbf F_q$,
    \[
     \#_{\mathbf F_q} E_{1,2}(\bm{d},\delta_1)  \ll_{\varepsilon , g }q^{d(1+15\delta_1 +\varepsilon)-d_{34}} . 
    \]
\end{lemma}


\begin{proof}[Proof of \cref{Le: E_11}] 

Starting again from \eqref{Eq: h0.bound.applied}, 
we have
\begin{align*}
h^0(\C, \mathcal{E}(\bm{a}',\bm{d}'',\bm{D},\bm{E}))
&
    \leqslant \sum_{i=1}^{r-1}
        \max \left \{\rk(\E_i/\E_{i-1})(\lambda_i+1-g),\rk(\E_i/\E_{i-1})g \right \} 
    + 2g-2 \\
&
    \leqslant 2 (\lambda_1+1-g) + 2g -2 = 2\lambda_1,
\end{align*}

and 
\begin{align*}
h^0(\C,\mathcal{E}(\bm{a}',\bm{d}'',\bm{D},\bm{E})
& \geqslant \rk(\E)(\mu(\E)+1-g) \\
& \geqslant d (1/5-13\delta_1) +3(1-g) .
\end{align*}
From these two inequalities, we deduce 
\[
\lambda_1
 \geqslant  
\frac{d(1-65\delta_1)}{10}+\frac{3(1-g)}{2}. 
\]

Therefore, when $k = \mathbf F_q$ we obtain that $| \#_{\mathbf F_q} 
    E_{1,1}(\bm{d},\delta_1) |$ is bounded above by
\begin{align}
\begin{split}\label{Eq: Estim.E11}
\sum_{\lambda\geq \frac{d(1-65\delta_1)}{10}+\frac{3(1-g)}{2}}q^{2\lambda} \cdot 
\#_{\mathbf F_q} \sum_{\bm{a}'\in \mathcal{H}^*_{\bm{d}'}}
    \sum_{\bm{D}, \bm{E}\in \DDiv(\C)^4_{\leqslant \delta_1 d}}
    \left [
        \bm{a}''\in \E (\bm{a}',\bm{d}'',\bm{D},\bm{E}) (2g-\lambda ) 
        \; \mid a_{34}\neq 0
    \right ] 
    \end{split}
\end{align}
It is easy to see that this bound also holds for the dimension. 
The point now 
is that if we set 
\begin{align*}
a_{14}&=\frac{a_2a_{24}-a_3a_{34}}{a_1},\\
a_{23}&=\frac{a_4a_{34}+a_1a_{13}}{a_2},\\
a_{12}&=\frac{a_3a_4a_{34}+a_1a_3a_{13}}{a_1a_2},
\end{align*}

then by \cref{Le: CongruenceImpliesPluecker} the vector $(\bm{a}',\bm{a}'')=(\bm{a}',a_{12},a_{13},\dots, a_{34})$ satisfies \eqref{equation-pluckers-relations} and 
\[
a_{ij}\in H^0(\C, L_{ij}(d_{ij}+2g-\lambda)).
\]
Since $a_{34}\neq 0$, we must have $\lambda\leqslant d_{34}+2g$ and hence $\lambda\leqslant d_{ij}+2g$ for all $(i,j)$, because $d_{34}$ is the minimum among them. We also have 
\[
(d_{ij}-\lambda)+d_j+(d_{jk}-\lambda)+d_k+(d_{kl}-\lambda)=d-3\lambda
\]
for every choice of indices. It follows that 
\begin{align*}
    &  \#_{\mathbf F_q} \sum_{\bm{a}'\in \mathcal{H}^*_{\bm{d}'}}
    \sum_{\bm{D}, \bm{E}\in \DDiv(\C)^4_{\leqslant \delta_1 d}}
    \left [
        \bm{a}''\in \E (\bm{a}',\bm{d}'',\bm{D},\bm{E}) (2g-\lambda ) 
        \; \mid a_{34}\neq 0
    \right ] 
    \\
    & \ll_\varepsilon 
        q^{d(2\delta_1 +\varepsilon)}
        \#_{\mathbf F_q}
        \left [
        ( \bm a' , \bm a '' ) \in \mathcal{H}^*_{\bm{d}'}
        \times \mathcal{H}_{\bm{d}'' + \bm{2g} - \bm \lambda}  \; \left| \; 
        \begin{array}{l}
\eqref{equation-pluckers-relations} \text{ and } a_{34}\neq 0
        \end{array}
        \right.
        \right ]
        \\
    & \ll_\varepsilon
        q^{d(1+2\delta_1 +\varepsilon)-3\lambda},
\end{align*}
where we used the divisor estimate for $E_1,\dots, E_4, D_3,D_4$ and estimated the sum over $D_1,D_2$ trivially. The last bound follows from \cref{Prop: GeneralUpperBound}. 
For a general base field $k$, the corresponding dimensional bound is simply $d ( 1 + 2 \delta_1 ) - 3 \lambda$ plus a constant depending linearly on $g$. 
In total, 
plugging this into \eqref{Eq: Estim.E11} we obtain 
\begin{align*}    
 \#_{\mathbf F_q} E_{1,1}(\bm{d},\delta_1) 
    & \ll_\varepsilon
        q^{d(1 +2\delta_1+\varepsilon)}
        \sum_{\lambda\geq \frac{d(1-65\delta_1)}{10}+\frac{(1-g)}{2}}
            q^{-\lambda}\\
    & \ll_\varepsilon
        q^{d(9+85\delta_1 +\varepsilon)/10 },
\end{align*}
and similarly (with no $\varepsilon$) for the dimensional bound,
which completes our proof of \cref{Le: E_11}.\end{proof}


\begin{proof}[Proof of \cref{Le: E_12}] We have to bound the contribution from $E_{1,2}(\bm{d},\delta_1)$.

In this case, we have $a_{34}=0$ for all $(a_{ij})\in H^0(\C, \E_1(2g-\lambda_1))$
(also, remember that $\lambda_1 > 2g - 2 $).
Since $\E_1$ is a subbundle of $\E(\bm{a}',\bm{d}'',\bm{D},\bm{E})$, any $(a_{ij})\in H^0(\C, \E_1(2g-\lambda_1))$ must satisfy 
\[
a_{ij}\in H^0(\C, L_{ij}(d_{ij}+2g-\lambda_1))
\]
as well as
\[
\ddiv(a_1)+F_{14}\leqslant \ddiv(a_2a_{24})\quad\text{and}\quad \ddiv(a_2)+F_{23}\leqslant \ddiv(a_1a_{13}).
\]
Since $(\ddiv(a_1);\ddiv(a_2))=0$, the last two conditions imply 
\[
\ddiv(a_1)\leqslant \ddiv(a_{24})\quad\text{and}\quad \ddiv(a_2)\leq\ddiv(a_{13}).
\]
In particular, if $a_{24}\neq 0$, then 
\begin{equation} \label{eq:boundlambda1E12a24neq0}
\lambda_1 - 2 g \leqslant 
d_{24}-d_1 
\end{equation}
and similarly $a_{13}\neq 0$ implies 
\begin{equation}\label{eq:boundlambda1E12a13neq0}
 \lambda_1 - 2 g \leqslant d_{13}-d_2 .
\end{equation}

Let us now deal with the contribution to $E_{1,2}(\bm{d},\delta_1)$, where we can find $(a_{ij})\in H^0(\C,\E_1(2g-\lambda_1))$ satisfying $a_{24}\neq 0$. In this case, we must thus have \eqref{eq:boundlambda1E12a24neq0}. 
As $r\in \{2,3\}$ and $\lambda_r\leqslant 2g-2$,  \cref{Le: Upper.Bound.h0} implies
\[
h^0(\C, \E(\bm{a}',\bm{d}'',\bm{D},\bm{D}))\leqslant 2(\lambda_1+1-g)+2g-2 \leqslant 2(d_{24}-d_1 +g+1)+2g-2. 
\]
Using the divisor bound for $E_1,\dots, E_4, D_3,D_4$ and the trivial estimate for $D_1,D_2$, we deduce that when $k = \mathbf F_q$
the contribution 
to the counting measure
in this case is 
\begin{align*}
    \ll \#_{\mathbf F_q} 
    \left ( \sum_{\bm a' \in \mathcal{H}^*_{\bm{d}'}}\sum_{\substack{\bm{D}\in \DDiv(\C)^4 \\ \deg(D_i)\leqslant \delta_1 d}}\sum_{\substack{\bm{E}\in \DDiv(\C)^4 \\ \deg(E_i)\leqslant \delta_1 d}} \mathbf L^{2(d_{24}-d_1)} \right ) & \ll_\varepsilon  q^{d_2+d_3+d_4-d_1 + 2d_{24} +d(2\delta_1 +\varepsilon)}\\
    & \ll_\varepsilon  q^{d(1+2\delta_1 +\varepsilon)-d_{34}},
\end{align*}
where we used that 
\begin{align*}
d_2+d_3+d_4+2d_{24}-d_1 
    & = 2d_0 + d_3 - d_1-d_2-d_4 \\
    & = d+2d_3 - d_0 \leqslant d+d_3+d_4 -d_0 \\
    & = d-d_{34}. 
\end{align*}
This is clearly satisfactory for the bound claimed in \cref{Le: E_12}. The bound on the dimension is obtained by bounding the exponents of $\mathbf L$ in the same way, ignoring the divisor bound. 

Finally, it remains to deal with the contribution, where $a_{34}=a_{24}=0$ for all $(a_{ij})\in H^0(\C, \E_1(2g-\lambda_1))$, 
for which we again use the averaging idea that appeared in the treatment of $E_{1,1}(\bm{d},\delta_1)$. 
Note that since $h^0(\C, \E_1(2g-\lambda_1))>0$, this implies that there exists $(a_{ij})\in H^0(\C, \E_1(2g-\lambda_1))$ with $a_{13}\neq 0$ and hence \eqref{eq:boundlambda1E12a13neq0} holds.

Since $\E_1(2g-\lambda_1)$ is semi-stable of slope $\lambda_1+2g-\lambda_1=2g$, the second part of \cref{Le: Upperbound.h^0.semistable} tells us that it is generated by global sections and as $a_{34}=a_{24}=0$ for all $(a_{ij})\in H^0(\C, \E_1(2g-\lambda_1))$, it follows that $\rk(\E_1)=\rk(\E_1(2g-\lambda_1))=1$. As $r\in \{2,3\}$, this only leaves the possibilities that $r=2$ and $\E_2=\E(\bm{a}',\bm{d}'',\bm{D},\bm{E})$ or $r=3$ and $\rk(\E_2)=2, \rk(\E_3)=3$. 
In the first case, we have $\lambda_2\leqslant 2g-2$ by \cref{Le: h^1not0implieslambda_rsmall}
and hence combining with \eqref{eq:boundlambda1E12a13neq0}
we get 
\[
h^0(\C, \E(\bm{a}',\bm{d}'',\bm{D},\bm{E}))\leqslant \lambda_1 +1-g +\lambda_2 \leqslant d_{13}-d_2 + 3g-1.
\]
Again applying the divisor bound for the sum over $E_1,E_2,E_3, E_4, D_1, D_3$ when $k = \mathbf F_q$ and estimating the sum over $D_2,D_4$ trivially shows that the contribution from this case is 
\[
\ll q^{d_1+d_3+d_4+d_{13} +d(2\delta_1 +\varepsilon)} \ll q^{d(1+2\delta_1 +\varepsilon)-d_{34}},
\]
where we used that 
\begin{align*}
d_1+d_3+d_4+d_{13} & = d_0 +d_4 \\
& = d-2d_0 +d_1+d_2+d_3+2d_4 \\
& = d+d_1 -d_{34}-d_{24} \\
& \leqslant d-d_{34},
\end{align*}
since $d_1\leqslant d_{24}$. 
Again, the dimensional bound is obtained by skipping the divisor bound and controlling the exponents of $\mathbf L$ the same way.
Therefore, the case under consideration is satisfactory for \cref{Le: E_12}. 

The only remaining case is thus when $r=3$, so that $\rk(\E_1)=\rk(\E_2/\E_1)=\rk(\E_3/\E_2)=1$. Observe that since $\E_1(2g-\lambda_1)$ is generated by global sections and every $(a_{ij}) \in H^0(\C, \E_1(2g-\lambda_1))$, satisfies $a_{34}=a_{24}=0$, it follows that $\E_1(2g-\lambda_1)$ injects into  
\[
L_{13}(d_{13}+2g-\lambda_1)\otimes \OK_\C(-\ddiv(a_2)) 
\]
and hence $\E_1$ injects into 
\[
L_{13}(d_{13})\otimes \OK_\C(-\ddiv(a_2) ).
\]

Since $\mu(\E_1(2g-\lambda_2))=\lambda_1+2g-\lambda_2> 2g-2$, it follows from \cref{Le: h^1not0implieslambda_rsmall} that $h^1(\C, \E_1(2g-\lambda_2))=0$.
Hence we have a short exact sequence
\[
\begin{tikzcd}[sep=small]
0 \rar & H^0(\C, \E_1(2g-\lambda_2))\rar & H^0(\C, \E_2(2g-\lambda_2)) \rar & H^0(\C, (\E_2/\E_1)(2g-\lambda_2))\rar & 0 .    
\end{tikzcd}
\]
The line bundle $(\E_2/\E_1)(2g-\lambda_2)$ is semi-stable of slope $2g$ and hence generated by global sections, by the second part of \cref{Le: Upperbound.h^0.semistable}.
As $\rk(\E_2/\E_1)=1$ and 
$
\E_1 
$
injects into $L_{13}(d_{13}+2g-\lambda_1)\otimes \OK_\C(-\ddiv(a_2))$, 
this implies that there exists $(a_{ij})\in H^0(\C, \E_2(2g-\lambda_2))$ such that $a_{34}\neq 0$. In addition, the space of such $(a_{ij})$ has dimension at least $h^0(\C, \E_1(2g-\lambda_2))= \lambda_1-\lambda_2+g+1$.

We have 
\begin{align*}
\deg(\E)+3(1-g) & \leqslant h^0(\C,\E) \\
& \leqslant \lambda_1+\lambda_2+2-g \\
& \leqslant d_{13}-d_2+2+g+\lambda_2,
\end{align*}
where the first inequality follows from Riemann--Roch; the second one from \cref{Le: Upper.Bound.h0} and the final one from $\lambda_1\leqslant d_{13}-d_2+2g$. In addition, by \cref{lemma-degree-main-vector-bundle} we have 
\begin{align*}
\deg(\E) & = d - d_1 - d_2 - d_3 - d_4 -b \\
& \geq d-d_1-d_2-d_3-d_4-13\delta_1 d
\end{align*}
and hence 
\begin{align*}
    \lambda_2 & \geq d-d_1-d_3-d_4-d_{13}-13\delta_1 d -4g \\
    & =d_{34}+d_{24}-d_1-13\delta_1 d -4g \\
    & \geq d_{34}-13\delta_1d-4g,
\end{align*}
where we used that $d=d_{13}+d_3+d_{34}+d_4+d_{24}$ and $d_1\leqslant d_{24}$. 
 
Therefore, the contribution from this case to $E_{1,2}(\bm{d},\delta_1)$ is bounded by 
\begin{align*}
 \#_{\mathbf F_q} E_{1,2}(\bm{d},\delta_1) & \ll \sum_{\substack{\lambda_2\geq d_{34}-13\delta_1d-4g\\ \lambda_2 \geq \lambda_1\geq  d_{13}-d_2+2g}}\sum_{\bm a' \in \mathcal{H}^*_{\bm{d}'}}\sum_{\substack{\bm{D}\in \DDiv(\C)^4 \\ \deg(D_i)\leqslant \delta_1 d}}\sum_{\substack{\bm{E}\in \DDiv(\C)^4 \\ \deg(E_i)\leqslant \delta_1 d}}\sum_{\substack{(a_{ij})\in \E(2g-\lambda) \\ a_{34}\neq 0}}q^{\lambda_1+\lambda_2}q^{-h^0(\C,\E_1(2g-\lambda_2))}\\
    & \leqslant d\sum_{\lambda\geq d_{34}-13\delta_1d-4g}q^{2\lambda}\sum_{\bm a' \in \mathcal{H}^*_{\bm{d}'}}\sum_{\substack{\bm{D}\in \DDiv(\C)^4 \\ \deg(D_i)\leqslant \delta_1 d}}\sum_{\substack{\bm{E}\in \DDiv(\C)^4 \\ \deg(E_i)\leqslant \delta_1 d}}\sum_{\substack{(a_{ij})\in \E(2g-\lambda) \\ a_{34}\neq 0}}1.
\end{align*}
Arguing exactly as in the paragraph below \eqref{Eq: Estim.E11}, one sees that this quantity is 
\[ O_\varepsilon
    (q^{d(1+15\delta_1+\varepsilon)-d_{34}}),
\]
which completes our proof of \cref{Le: E_12}.
\end{proof}
\section{Proofs of Theorems A and B}
We now have everything at hand to complete the proofs of our main theorems. We may assume that $\bm{d}\in \bf{N}^{10}$ satisfies the conditions from \cref{Prop:ConeDecomp} in what follows.

First, thanks to \cref{Prop: Passing.UT} 
we have 
\begin{equation} \label{eq: VariantPropPassing.UT}
    \left [
        \Hom_k^{\bm d} ( \mathscr C , X )_U 
   \right ] 
   =
   \frac{1}{ ( \mathbf L_k - 1 )^5}
    \sum_{
        \substack{
        \underline L\in \PPic^0(\C)^{10}
        \\
        \otimes_{i\in \mathfrak{I}} L_i^{\otimes n_i} \simeq \mathcal O_\mathscr C 
        \forall \bm n \in N_X
        }
    }
   [ M ( \bm d , \underline L ) ] 
\end{equation}
in $\mathscr M_k [ ( \mathbf L - 1 )^{-1}]  \hookrightarrow \widehat{\mathscr M_k}^{\dim}$. 

Thanks to \cref{lemma-motivic-functions-are-defined-on-points}, we can fix a class $\underline L\in \PPic^0(\C)^{10}$ and forget the triviality condition for a moment. Next, \eqref{Eq: LargeMoebiusRemoved} allows us remove the contribution from large M\"obius variables. For any $0<\delta <1$, we have 
\[
M(\bm{d},\underline{L})
=
\sum_{
    \bm{a}'\in \mathcal{H}^*_{\bm{d}'} ( \underline L ' ) 
    }
    M(
        \bm{a}',\bm{d},\underline{L}, \delta 
    )
    + E (\delta),
\]
where 
\begin{equation}\label{Eq: Final.E}
    \dim (E(\delta))\leq d(1-\delta) \quad \text{and}\quad \#_{{\bf F}_q}(E(\delta))\ll_\varepsilon q^{d(1-\delta+\varepsilon)}.
\end{equation}
Then we decompose the first term into 
\[
\sum_{
    \bm{a}'\in \mathcal{H}^*_{\bm{d}'} ( \underline L ' )
    }
    M(
        \bm{a}',\bm{d},\underline{L} ,\delta 
    )
= 
M_1(\bm{d},\delta)+E_0(\bm{d},\delta)-E_1(\bm{d},\delta),
\]
where $E_0(\bm{d},\delta)$ and $E_1(\bm{d},\delta)$ are defined in \eqref{Eq: ErrorTermI.MainBody} and \eqref{Eq: Defi.E1}.

\subsection{The main contribution}
For $M_1(\bm{d},\delta)$, we have 
\[
M_1(\bm{d},\delta)
= 
{\bf L}^{d-|\bm d'|+3(1-g)}
\sum_{\bm{a}'\in H^*_{\bm{d}'}(\underline{L}')}\theta(\ddiv(\bm{a}'))+E_2(\delta),
\]
where 
\begin{equation}\label{Eq: E2.final}
    \dim(E_2(\delta))\leq d(1-\delta/2) +O_g(1)\quad\text{ and }
    \quad
    \#_{\mathbf F_q}(E_2(\delta))\ll q^{d(1-\delta/2)},
\end{equation}
which is a straightforward consequence of  \cref{Eq: MainTerm.RemoveTruncation}. 

Now we want to use \cref{Prop: EvalMainTerm}.
First, we have the relation
\[
\sum_{\bm{a}'\in H^*_{\bm{d}'}(\underline{L}')}\theta(\ddiv(\bm{a}')) {\bf L}^{-| \bm d ' |}
= 
( \mathbf L - 1 )^4 
\sum_{
\substack{
\bm A ' \in \DDiv^{\bm d '} ( \mathscr C ) 
\\
A_i \in [L_i + d \OK_\C ( \mathfrak{D}_1 )]
}
}
\theta_0 ( \bm A '  )
\theta ( \bm A '  ) {\bf L}^{-| \bm d ' |} . 
\]
Summing over $L_1 , ... , L_4\in \PPic^0(\C)$ gives 
\begin{align*}
 & \sum_{L_1 , ... , L_4 \in \PPic^0 ( \C ) } \sum_{\bm{a}'\in H^*_{\bm{d}'}(\underline{L}')}\theta(\ddiv(\bm{a}')) {\bf L}^{-| \bm d ' |}\\
& =
\sum_{L_1 , ... , L_4 \in \PPic^0 ( \C ) }
( \mathbf L - 1 )^4
\sum_{
\substack{
\bm A ' \in \DDiv^{\bm d '} ( \mathscr C ) 
\\
A_i \in [L_i + d \OK_\C ( \mathfrak{D}_1 )]
}
}
\theta_0 ( \bm A '  )
\theta ( \bm A '  ) {\bf L}^{-| \bm d ' |} \\
& = ( \mathbf L - 1 )^4 
\sum_{
\bm A ' \in \DDiv^{\bm d '} ( \mathscr C ) 
}
\theta_0 ( \bm A '  )
\theta ( \bm A '  ) {\bf L}^{-| \bm d ' |}.
\end{align*}
Let us consider the map 
\[
\Hom^{\bm d}_k ( \mathscr C , X )_U 
\longrightarrow \PPic^0 ( \mathscr C )^4 
\]
sending a morphism $f : \mathscr C \to X$ to the linear classes of $f^* E_1 , ... , f^* E_4$ (via the identification $\PPic^0 ( \mathscr C )^4 \simeq \PPic^{d_1} ( \mathscr C ) \times_k ... \times_k \PPic^{d_4} ( \mathscr C )$ induced by the degree one divisor $\mathfrak D_1$). 
Since the exceptional divisors $E_1 , ... , E_4$ are linearly independent, this map is surjective and compatible with \eqref{eq: VariantPropPassing.UT}. 
The point is that we can write 
\[
\sum_{
        \substack{
        \underline L\in \PPic^0(\C)^{10}
        \\
        \otimes_{i\in \mathfrak{I}} L_i^{\otimes n_i} \simeq \mathcal O_\mathscr C 
        \forall \bm n \in N_X
        }
    }
=
\sum_{\underline L ' \in \PPic^0 ( \mathscr C )^4 }
\sum_{
        \substack{
        \underline L '' \in \PPic^0(\C)^6
        \\
        \otimes_{i\in \mathfrak{I}} L_i^{\otimes n_i} \simeq \mathcal O_\mathscr C 
        \forall \bm n \in N_X
        }
    }
\]
where the second sum boils down to summing over one single copy of $\PPic^0 ( \C ) $. In particular, 
if we sum a motivic function that does not depend on $\underline L''$, this simply gives an additional factor of $[ \PPic^0 ( \C ) ]$. Applying this our situation, up to the error $E_2(\delta)$ we get 
\begin{align*}
 \sum_{
        \substack{
        \underline L\in \PPic^0(\C)^{10}
        \\
        \otimes_{i\in \mathfrak{I}} L_i^{\otimes n_i} \simeq \mathcal O_\mathscr C 
        \forall \bm n \in N_X
        }
    }
M_1 ( \bm d , \underline L )
& =
\sum_{\underline L ' \in \PPic^0 ( \mathscr C )^4 }
\sum_{
        \substack{
        \underline L '' \in \PPic^0(\C)^6
        \\
        \otimes_{i\in \mathfrak{I}} L_i^{\otimes n_i} \simeq \mathcal O_\mathscr C 
        \forall \bm n \in N_X
        }
    }
{\bf L}^{d-|\bm d'|+3(1-g)}
\sum_{\bm{a}'\in H^*_{\bm{d}'}(\underline{L}')}\theta(\ddiv(\bm{a}'))\\
& =
\sum_{\underline L ' \in \PPic^0 ( \mathscr C )^4 }
[ \PPic^0 ( \mathscr C ) ]
{\bf L}^{d-|\bm d'|+3(1-g)}
\sum_{\bm{a}'\in H^*_{\bm{d}'}(\underline{L}')}\theta(\ddiv(\bm{a}')) \\
& =  \mathbf L^{d + 2 ( 1 - g )} [ \PPic^0 ( \mathscr C ) ] \mathbf L^{-g} \mathbf L ( \mathbf L - 1 )^4 \sum_{
\bm A ' \in \DDiv^{\bm d '} ( \mathscr C ) 
}
\theta_0 ( \bm A '  )
\theta ( \bm A '  ) {\bf L}^{-| \bm d ' |} .
\end{align*}
If we divide the last expression by $( \mathbf L - 1 )^5$, we obtain 
\[
\mathbf L^{d + 2 ( 1 - g )} 
\frac{
[ \PPic^0 ( \mathscr C ) ] \mathbf L^{-g}}
{1 - \mathbf L^{-1}}
\sum_{
\bm A ' \in \DDiv^{\bm d '} ( \mathscr C ) 
}
\theta_0 ( \bm A '  )
\theta ( \bm A '  ) {\bf L}^{-| \bm d ' |} 
\]
and \cref{Prop: EvalMainTerm} tells us that there exists $\eta'\in (0,1/2)$ such that this expression is
\begin{equation}\label{Eq: MainTerm.final}
\mathbf L^{d+2(1-g)} \left ( 
\frac{\left [ \PPic^0 ( \mathscr C ) \right ] \mathbf L^{-g}}{1 - \mathbf L^{-1}}
\right )^5 
\prod_{v \in \mathscr C}
\left ( 
1 - \mathbf L_v^{-1}
\right )^5
\left ( 
    1 + 5 \mathbf L_v^{-1} + \mathbf L_v^{-2}
\right ) +E_3(\eta'), 
\end{equation}
where 
\begin{equation}\label{Eq: E3.final}
    \dim(E_3(\eta'))\leq d-\eta' d_1 +O_g(1) \quad \text{and} \quad \#_{{\bf F}_q}(E_3(\eta'))\ll q^{d-\eta' d_1}.
\end{equation}
The main term agrees with \cref{Thm:TheTheorem.Fq(C)} and \cref{Th: TheTheorem.Motivic}. We are therefore left with showing that the error terms that arose in our deductions make up an acceptable contribution. 

\subsection{The error terms}
It follows from our discussion so far that 
\begin{equation}\label{shalalalala}
[\Hom^{\bm{d}}_k(\C,X)_U]=\mathbf{L}^{d+2(1-g)}c_0 +\frac{[\PPic^0(\C)]^5}{(1-\mathbf{L})^{4}}E'(\delta) +E_3(\eta'), 
\end{equation}
where $c_0$ is the constant appearing in \eqref{Eq: MainTerm.final} and
\[
E'(\delta)=E(\delta)+E_0(\bm{d},\delta)-E_1(\bm{d},\delta) +E_2(\delta).
\]
If we set
\[
e(\delta)=\max\left( d(1-\delta), d-d_{34}+8\delta d, d(9+150\delta)/10, d(1+2\delta)-d_{34}, d(1-\delta/2)\right),
\]
then it follows from \eqref{Eq: Final.E}, \cref{Prop: Bound.E0}, \cref{Prop: Error.E_1} and \eqref{Eq: E2.final} that
\begin{equation}
    \dim(E(\delta))\leq e(\delta)+O_g(1)\quad\text{and}\quad \#_{{\bf F}_q}(E(\delta)) \ll_{\varepsilon}q^{e(\delta)+ \varepsilon},
\end{equation}
assuming that $0<\delta<1/65$. We can now choose $\delta$ and $\varepsilon$ sufficiently small, to deduce with a straightforward computation that there exists $\eta>0$ such that 
\[
e(\delta)\leq d-\eta d_{34}\leq d-\eta d_1,
\]
since $d_1\leq d_{34}$. Replacing $\eta$ by $\eta'$ if necessary, we may now deduce from \eqref{shalalalala} and \eqref{Eq: E3.final} that 
\[
[\Hom_k^{\bm{d}}(\C,X)_U]={\bf L}^{d+2(1-g)}c_0 + G(\eta),
\]
where 
\[
\dim(G(\eta))\leq d-\eta d_1 +O_g(1)\quad \text{and}\quad \#_{\bf{F}_q}(G(\eta))\ll q^{d-\eta d_1}.
\]
This is satisfactory for \cref{Thm:TheTheorem.Fq(C)} and \cref{Th: TheTheorem.Motivic}, since \[
\text{dist}(\bm{d}, \partial \text{Eff}(X)^\vee)=\min_{i,j\in \{1,2,3,4\}}(d_i, d_{ij})=d_1.
\]

\bibliography{main}

\end{document}